\documentclass[twoside]{article}
\usepackage[letterpaper, marginpar=0pt, left=25.4mm, right=25.4mm, bottom=32mm, headsep=3.6mm, footnotesep=6mm]{geometry}
\usepackage{amsfonts, amssymb, amsthm, mathtools, mathrsfs, enumitem, tikz, graphics, fancyhdr}
\newtheorem{thm}{\sc Theorem}[section]
\newtheorem{prop}[thm]{\sc Proposition}
\newtheorem{lemma}[thm]{\sc Lemma}
\newtheorem{cor}[thm]{\sc Corollary}

\theoremstyle{definition}
\newtheorem{defn}[thm]{\sc Definition}
\newtheorem{remark}[thm]{\sc Remark}
\newcommand{\N}{\mathbb{N}}
\newcommand{\Z}{\mathbb{Z}}
\newcommand{\Q}{\mathbb{Q}}

\newcommand{\C}{\mathbb{C}}
\newcommand{\uhp}{\mathbb{H}}

\newcommand{\oK}{\mathscr{O}_K}
\newcommand{\ok}{\mathscr{O}_k}
\newcommand{\oL}{\mathscr{O}_L}
\newcommand{\id}{\mathbb{I}_2}
\newcommand{\br}{\hspace{-0.6667em}}
\newcommand{\isom}{\overset{\raisebox{0.2em}{\scriptsize$\cong$\hspace{0.15em}}}{\,\smash\longrightarrow\,}}
\newcommand{\intro}[1]{\begin{center}\begin{minipage}{\dimexpr\textwidth-4cm}\textbf{\scalebox{0.825}[0.975]{\footnotesize{A}\scriptsize{BSTRACT}.}}  \small{#1}\end{minipage}\end{center}}
\newcommand{\auth}{\author{\normalsize{A}\footnotesize{RAV}\normalsize{ V. K}\footnotesize{ARIGHATTAM} \vspace{0.3cm}}}

\DeclareMathOperator{\ord}{ord}
\DeclareMathOperator{\Frob}{\mathsf{Frob}}
\DeclareMathOperator{\GL}{GL}
\DeclareMathOperator{\SL}{SL}
\DeclareMathOperator{\Cl}{Cl}
\DeclareMathOperator{\Mat}{\mathsf{Mat}}
\DeclareMathOperator{\Gal}{Gal}
\DeclareMathOperator{\Tr}{Tr}
\DeclareMathOperator{\Res}{\mathsf{Res}}
\DeclareMathOperator{\rk}{rk}
\DeclareMathOperator{\lcm}{lcm}
\let\ker\relax\DeclareMathOperator{\ker}{\mathsf{ker}}
\DeclarePairedDelimiter{\abs}{\lvert}{\rvert}
\def\Im{\operatorname{Im}}
\def\Re{\operatorname{Re}}
\makeatletter
    \def\@seccntformat#1{\@ifundefined{#1@cntformat}
        {\csname the#1\endcsname\quad}
        {\csname #1@cntformat\endcsname}}
    \newcommand\section@cntformat{\sc\large\thesection.\;\;\;}
    \newcommand\subsection@cntformat{\sc\S\:\!\thesubsection.\;\;\;}
    \let\s@ction\section
    \let\subs@ction\subsection
    \def\section{\@ifstar{\@sectionstar}{\@section}}
    \newcommand{\@section}[1]{\s@ction{\centering\large\sc#1}\vspace{0.3em}}
    \newcommand{\@sectionstar}[1]{\s@ction*{\centering\large\sc#1}}
    \renewcommand{\subsection}[1]{\subs@ction{\normalfont #1.}\medskip}
\makeatother
\allowdisplaybreaks[4]
\numberwithin{equation}{section}
\title{Heegner point constructions and fundamental units in cubic fields}
\auth
\date{ }
\begin{document}
\maketitle
\intro{We use Heegner points to prove the existence of nontorsion rational points on the elliptic curve $y^2 = x^3 + D$ for any rational number $D=a/b$ such that $a$ and $b$ are squarefree integers for which $6$, $a$, and $b$ are pairwise relatively prime, $a\equiv b\!\pmod{4}$, $\abs*{a}\abs*{b}^{-1}\equiv5$ or $7\!\pmod{9}$, and $h_K$ is odd, where $K\coloneqq\Q(\!\sqrt[3]{\hspace{-0.05em}D})$.  In particular, we show that under these assumptions, the elliptic curve with equation $y^2 = x^3 + D$ has algebraic rank $1$ and the elliptic curve with equation $y^2 = x^3 - D$ has algebraic rank $0$.  This follows from our new expression for the fundamental unit of $\oK$ in terms of the class number $h_K$ and the norm of a special value of a modular function of level 6, for any integer $D$ relatively prime to $6$, not congruent to $\pm1\!\pmod{9}$, for which no exponent in its prime factorization is a multiple of $3$.  This expression is an analogue of a theorem of Dirichlet in 1840 relating the fundamental unit of a real quadratic field to its class number and a product of cyclotomic units.}
\vspace{0.6em}

\section{Introduction.}
Consider the fundamental question of determining when a rational number $D$ can be expressed as the sum of a rational square and a rational cube, or more specifically, when the elliptic curve $E_D$ with (affine) equation $y^2 = x^3 + D$ has algebraic rank $1$ over $\Q$.  As an example, if $D = 7$, there are no nontrivial rational points on $E_D$ (see \cite{Silv1}, Remark XI.7.1.1), while if $D = 5$, the group $E_D(\Q)$ includes the point $(x,y)=(1,2)$ among infinitely many others.  There has been much progress concerning the related question for the elliptic curve $\widetilde{E}_k$ with equation $X^3 + Y^3 = k$ (note that the curves $\widetilde{E}_{2k}$ and $E_{-27k^2}$ are isomorphic over $\Q$ for all $k$).  For instance, $\widetilde{E}_k$ has nontrivial points when $k/2$ is either a prime number congruent to $2\!\pmod{9}$ or the square of a prime number congruent to $5\!\pmod{9}$ (Satg\'e \cite{Sa}).  Recent work has shown that $\widetilde{E}_k$ has algebraic rank $1$ in several other cases when $k$ has either $1$ or $2$ prime factors, such as when $k$ is a prime number congruent to $4$ or $7\!\pmod{9}$ (which is the subject of Sylvester's problem, see Elkies \cite{E} and Dasgupta and Voight \cite{DV, DV2}), and when $k$ is congruent to $3$ times a prime or the square of a prime congruent to $2$ or $5\!\pmod{9}$ (Shu, Song, and Yin \cite{SSY}, Shu and Yin \cite{SY2}).  For the congruent number problem, which concerns elliptic curves $E/\Q$ with $j$-invariant $1728$, Tian \cite{Tian} has found more general conditions under which these elliptic curves have algebraic and analytic rank $1$.

In this paper, we will use Heegner points to construct nontrivial rational points on the elliptic curves $E_D$ for a large class of rational numbers $D$ that can be expressed as the quotient of relatively prime squarefree integers.  Note that it is a simple exercise using the injectivity of reduction maps (see Silverman \cite{Silv1}, Exercise 10.19) to see that $E_D(\Q)_{tors}$ is trivial unless $D$ is a square, a cube, or $-432$ times a sixth power (none of which hold for the values of $D$ we will be considering).  Our main result is the following.
\begin{thm}\label{mainresult} Suppose that $a$ and $b$ are squarefree integers such that $6$, $a$, and $b$ are pairwise relatively prime and $\abs*{a}\abs*{b}^{-1}\equiv5$ or $7\!\pmod{9}$.  Let $D\coloneqq a/b$ and let $\epsilon\coloneqq(-1)^{(a-b)/2}$.  If the class number $h_K$ of the cubic field $K\coloneqq\Q(\!\sqrt[3]{\hspace{-0.05em}D})$ is odd, the elliptic curve $E_{\epsilon D}$ with equation $y^2 = x^3 + \epsilon D$ has algebraic rank $1$ over $\Q$, and the elliptic curve $E_{-\epsilon D}$ with equation $y^2 = x^3 - \epsilon D$ has algebraic rank $0$ over $\Q$.
\end{thm}
\begin{remark} While we prove that the curve $E_{\epsilon D}$ has algebraic rank $1$ using Heegner point constructions, the fact that $E_{-\epsilon D}$ has algebraic rank $0$ follows from the following bound due to Cassels \cite{Cas}:  If $N$ is an odd integer not congruent to $\pm1\!\pmod{9}$ such that no exponent in the prime factorization of $N$ is divisible by $3$ and the class number of $\Q(\!\sqrt[3]{N})$ is odd, the sum of the ranks of $E_N(\Q)$ and $E_{-N}(\Q)$ is at most $1$.
\end{remark}

Let $\omega\coloneqq e^{2\pi i/3}$, $n=\abs{ab^5}$, and let $R_n$ denote the ring class field of the order $\Z[n\omega]\subseteq\Z[\omega]$.  In Section \ref{modparam}, we construct modular functions $X,Y\colon X(6)\to\mathbb{P}^1$ in terms of division values of the Weierstrass $\wp$-function that form a modular parametrization $\phi\coloneqq(X,Y)\colon X(6)\isom E_1$; this approach was taken by Monsky \cite{M1,M2} for constructing a modular parametrization for the curve $2y^2 = x^4 + 1$.  In order to construct nontrivial rational points on $E_{\epsilon D}$, we start with the point $\phi(n\omega)\in E_1(R_{6n})$, and consider its image in $\widetilde{E}_2$ under an isomorphism over $\Q(\omega)$.  We then generalize the methods in Satg\'e \cite{Sa} to perform a cubic twist and obtain a point in $\smash{\widetilde{E}_{2n^2}}(R_n)$.  Since $E_D$ is a twist of $E_1$ of degree $6$, we still need to perform a quadratic twist in order to obtain a rational point on $E_D$.  If we define $\rho\coloneqq (-1)^{(n-1)/2}$, we note that $E_{\epsilon D}\cong E_{\rho n}$ over $\Q$, $E_{\rho n}\cong E_{-27\rho n}$ over $\Q(\omega)$, and $E_{-27\rho n}$ is a quadratic twist of the elliptic curve $\smash{\widetilde{E}_{2n^2}}$ over $\Q(\!\sqrt{\smash{\hspace{-0.05em}\rho n}\vphantom{2}})$.  The final step in our construction is to take the image of our point in $\smash{\widetilde{E}_{2n^2}}(R_n)$ under these isomorphisms, which yields a point in $E_{\epsilon D}(R_n)$, and then take the trace over $R_n(\Q)$ to obtain a point $S\in E_{\epsilon D}(\Q)$.  We will show that the point $S$ is nontrivial, and also determine additional information on how the point $S$ fits in the group $E_{\epsilon D}(\Q)$.  For simplicity, our arguments in Section \ref{heegner} are expressed in terms of the point $S^*\in E_{\rho n}(\Q)$ (defined formally in equation (\ref{thepointS})) which maps to $S$ under the isomorphism $E_{\rho n}\smash{\isom}E_{\epsilon D}$.

\begin{thm}\label{rankgenerator} In the setting of Theorem \ref{mainresult}, the Heegner point trace $S\in E_{\epsilon D}(\Q)$ is an odd multiple of the generator of $E_{\epsilon D}(\Q)$.
\end{thm}

The complexity in proving Theorem \ref{mainresult} lies in the nontriviality argument for $S^*\in E_{\rho n}(\Q)$.  The most important tool in this argument is our new identity relating the norm of a special value of the modular function $X(\tau)+1$ to the fundamental unit of an associated cubic field.

\begin{thm}\label{fundunitcube} Let $n$ be any positive integer relatively prime to $6$ such that $n\not\equiv\pm1\!\pmod{9}$ and no exponent in the prime factorization of $n$ is a multiple of $3$, let $u$ be the (unique) fundamental unit of the cubic field $K\coloneqq\Q(\!\sqrt[3]{\vphantom{2}\hspace{-0.05em}n})$ greater than $1$, and let $X\colon X(6)\to\mathbb{P}^1$ be the modular function of weight $0$ defined in Section \ref{modparam}. Then $X(n\omega)\in R_{6n}$, and
\[N_{R_{6n}/K}(X(n\:\!\omega)+1)=3^{f(n)}u^{3h_K\sigma(n/n')},\]
where $n'$ is the largest squarefree divisor of $n$ and\[f(n)\coloneqq n\prod_{p\mid n}\Big[1-p^{-1}\Big(\frac{p}{3}\Big)\Big].\]\end{thm}
\begin{remark} This theorem is an analogue of a theorem of Dirichlet \cite{E, LD} from 1840 expressing the fundamental unit of a real quadratic field in terms of the class number of that field and a norm over a cyclotomic field extension.  In particular, a special case of that theorem states that if $D>1$ is a squarefree integer such that $D\equiv1\!\pmod{4}$, $L\coloneqq\Q(\sqrt{D})$, $u$ is the (unique) positive fundamental unit of $\oL$ greater than $1$, and $a$ is a quadratic nonresidue modulo $D$, 
\[N_{\Q(\zeta_D)/L}\bigg(\frac{1-\zeta_D^a}{1-\zeta_D}\bigg)=u^{2h_L},\]
where $\zeta_D\coloneqq e^{2\pi i/D}$.   Our theorem is also inspired by the ability to determine the prime factorization of special values of modular functions \cite{L}.\label{dirichlet}\end{remark} 

The proof of Theorem \ref{fundunitcube} has two main components.  The first, Theorem \ref{harmonicfns}, relates the logarithm of an appropriate product of translates by elements of $\SL_2\Z$ of the modular function $X(\tau)+1$ to the limit as $s\to1$ of a weighted sum of non-holomorphic Eisenstein series of the form $E_{t\in X(3n)}(\tau,s)$.  The second, Theorem \ref{lfunction}, rewrites the desired special value of the weighted sum of non-holomorphic Eisenstein series in terms of the Riemann zeta function and a Dirichlet $L$-function over $\Q(\omega)$ associated to a cubic character.  The remaining ingredients are the class number formula, which is central to the proof of Dirichlet's theorem (see Remark \ref{dirichlet}), and the Shimura Reciprocity Theorem \cite{Shim}, which we use to rewrite the norm of $X(n\omega)+1$ as the value at $\omega$ of a modular function on $X(6n)$.  

In Sections \ref{modparam} and \ref{nonhol}, we provide the necessary preliminaries in modular forms and non-holomorphic Eisenstein series.  We prove Theorem \ref{fundunitcube} in Section \ref{galconj}, and Theorems \ref{mainresult} and \ref{rankgenerator} in Section \ref{heegner}. \vspace{0.3em}
\section{\label{modparam}Modular Functions.}
Let us review some results about modular functions for congruence subgroups of $\SL_2\Z$ and their zeros and poles, and then construct modular forms $X$ and $Y$ of level $6$ parametrizing the elliptic curve $E_1$ with equation $y^2 = x^3 + 1$.  After analyzing their properties, we will state the Shimura Reciprocity Theorem \cite{Shim}, which serves an essential role in the Heegner point method by enabling us to determine the field of definition of a special value of a modular function.  We refer the reader to Shimura \cite{Shim}, Lang \cite{L}, Silverman \cite{Silv2}, and Iwaniec \cite{I1} for more details about modular functions.

Let $\Gamma\le\SL_2\Z$ be a congruence subgroup, i.e., a subgroup containing $\Gamma(N)$ for some integer $N$, and let $S$ and $T$ denote the inversion and translation matrices, respectively.  Let $s\in\uhp^*$ be a cusp and let $\gamma=\begin{bsmallmatrix}\!a&b\!\\\!c&d\!\end{bsmallmatrix}\in\SL_2\Z$ be a matrix such that $s=\gamma\infty$.  If $k$ is an even integer and $f$ is a modular function of weight $k$ for $\Gamma$, define $f_\gamma(\tau)\coloneqq(c\tau+d)^{-k}f(\gamma\tau)$.  Then $f_\gamma$ is a modular function of weight $k$ for the group $\gamma^{-1}\Gamma\gamma$.  Let $m$ be the smallest positive integer such that $T^m\in\gamma^{-1}\Gamma\gamma$; then $f_\gamma$ can be expressed as a function of $q_m\coloneqq e^{2\pi i\tau/m}$.

\begin{defn} If $f$ is a modular function of weight $k$, then the \textit{order} of $f$ at the cusp $s$ is
\[\ord_{s\,\in X(\Gamma)} f(\tau) \coloneqq \ord_{q_m=0} f_\gamma(q_m).\]
\end{defn}

Let $\mathcal{F}_\Gamma$ be a fundamental domain for the (left) action of $\Gamma$ on $\uhp$, and let $n_\Gamma\coloneqq[\SL_2\Z:\{\pm1\}\cdot\Gamma]$.  Let $S_\Gamma$ denote a set of inequivalent cusps for $\Gamma$, and let $S_i$ denote the set of inequivalent elliptic points of order $i$ (i.e., points where the covering map $\uhp\to\Gamma\backslash\uhp$ is branched with order $i$).  Then the poles and zeros of any modular function of weight $k$ for $\Gamma$ satisfy the following identity (which is Corollary I.3.8 in \cite{Silv2} for the case $\Gamma=\SL_2\Z$, and is the synthesis of Propositions 1.40 and 2.16 in \cite{Shim} in general).

\begin{prop}\label{zeropolesum} \textup{\cite{Shim, Silv2}} If $f$ is a modular function of weight $k$ for $\Gamma$,
\[\sum_{s\in S_\Gamma}\ord_s f(\tau) + \frac{1}{2}\sum_{s\in S_2}\ord_s f(\tau) + \frac{1}{3}\sum_{s\in S_3}\ord_s f(\tau) + \sum_{\rho\in\mathcal{F}_\Gamma\setminus(S_2\cup S_3)}\ord_\rho f(\tau) = \frac{n_\Gamma k}{12}.\]
\end{prop}

An important consequence of this proposition is that a (holomorphic) modular form of weight $k$ for $\Gamma$ has no zeros away from cusps and elliptic points whenever the first three sums in the proposition add to $n_\Gamma k/12$.  This will enable us to construct modular forms for which we have complete knowledge of the zeros and poles and their orders.  By Proposition \ref{nozeros} (below), we can also prove the equality of two different modular functions up to a constant factor only by using the orders of their zeros and poles.  Modular functions of weight $0$ are meromorphic functions between the Riemann surfaces $X(\Gamma)\coloneqq \Gamma\backslash\uhp^*$ (the modular curve associated to $\Gamma$) and $\mathbb{P}^1(\C)$, which correspond to algebraic morphisms between the corresponding complex projective curves (see, for example, Donaldson \cite{D}, \S11).  The orders of modular functions of weight $0$ at non-elliptic points of $X(\Gamma)$ as we have defined above are the same as the orders of the functions viewed as algebraic morphisms of complex projective curves. The following proposition is then a direct consequence of basic properties of algebraic morphisms between projective curves.

\begin{prop}\label{nozeros} Let $f$ be a modular function of weight $0$ for $\Gamma$ and suppose that $f$ has no poles in $\uhp^*$.  Then $f$ must be identically constant.
\end{prop}

For the applications of interest, we will consider the congruence subgroups
\begin{align*}
\Gamma(N)&\coloneqq\{\gamma\in\SL_2\Z:\gamma\equiv\id\br\pmod{N}\}\intertext{and}
\Gamma_0(N)&\coloneqq\Big\{\gamma=\begin{bmatrix}a&b\\c&d\end{bmatrix}\in\SL_2\Z: N\mid c\Big\}.
\end{align*}
The associated modular curves $X(\Gamma)$ are commonly denoted $X(N)$ and $X_0(N)$, respectively.  The subgroups $\Gamma(N)$ and $\Gamma_0(N)$ have indices (see Proposition 1.43 of \cite{Shim})
\begin{align*} [\SL_2\Z:\{\pm1\}\cdot\Gamma(N)]&=\begin{cases} N(2N-1) & \text{if }N \le 2 \\ \frac{N^3}{2}\prod_{p\mid N}\big(1-\frac{1}{p^2}\big) & \text{if }N>2
\end{cases}
\intertext{and}
[\SL_2\Z:\{\pm1\}\cdot\Gamma_0(N)]&=N\prod_{p\mid N}\big(1+\tfrac{1}{p}\big).\end{align*}
For these groups $\Gamma$, we will need to choose the sets $S_\Gamma$, $S_2$, and $S_3$ in order to apply Proposition \ref{zeropolesum} to a modular function for $\Gamma$.  It turns out that when $6\mid N$, the action of $\Gamma_0(N)$ (and hence of $\Gamma(N)$) on $\uhp^*$ has no elliptic points (\cite{Shim}, Proposition 1.42). The following is a set of inequivalent cusps for the action on $\uhp^*$ by $\Gamma(N)$ and $\Gamma_0(N)$ (see Section 2.4 of \cite{I1}).
\begin{prop}\textup{\cite{I1}}\label{inequivcusps} Let $N$ be a positive integer.\begin{enumerate}[label=\textup{(\alph*)}]
\item For each $\ell,m\in\Z$ such that $0\le\ell\le N/2$, $1\le m\le N$, and $\gcd(\ell,m,N)=1$, choose $\alpha_{\ell,m}\equiv\ell\!\pmod{N}$ and $\beta_{\ell,m}\equiv m\!\pmod{N}$ such that $\gcd(\alpha_{\ell,m},\beta_{\ell,m})=1$. 
Then the set \[\{\alpha_{\ell,m}/\beta_{\ell,m}: \gcd(\ell,m,N)=1\text{ and }0<\ell<N/2, 1\le m\le N\text{ or }\ell\in\{0, N/2\}, N/2\le m\le N\}\] is a set of inequivalent cusps for $\Gamma(N)$.  We will let $[\ell/m]$ denote the cusp $\alpha_{\ell,m}/\beta_{\ell,m}\in X(N)$.
\item For each $\ell,m\in\Z$ such that $m\mid N$ and $0\le\ell<\gcd(m,N/m)$, choose $\alpha_{\ell,m}\equiv\ell\!\pmod{\!\gcd(m,N/m)}$ such that $\gcd(\alpha_{\ell,m},m)=1$.  Then the set \[\{\alpha_{\ell,m}/m: m\mid N, 0\le \ell<\gcd(m, N/m)\}\] is a set of inequivalent cusps for $\Gamma_0(N)$.
\end{enumerate}
\end{prop}

Next, we consider how the properties of a modular function change when acting via a matrix $M$ of determinant greater than $1$.  In the following proposition, part (a) is based on Lang \cite{L}, Theorem 11.3, and we provide proofs for parts (b-d).

\begin{prop} Let f be a modular function for $\Gamma(N)$, let $M\in\Mat_{2\times2}\Z$ be a matrix with determinant $n>0$, and let $g(\tau)\coloneqq f(M\tau)$.\label{matrixtransform}\begin{enumerate}[label=\textup{(\alph*)}]
\item\label{gismodular} The function $g$ is a modular function for $\Gamma(nN)$.
\item\label{matrixuppertri} If $s$ is a cusp of $\uhp^*$, and $\gamma, \gamma'\in\SL_2\Z$ are matrices such that $\gamma\infty = s$ and $\gamma'\infty = Ms$, the matrix $(\gamma')^{-1}M\gamma=\begin{bsmallmatrix}\!A&B\!\\\!0&D\!\end{bsmallmatrix}$ is upper-triangular.
\item\label{orderofg} In the notation of part (b), $\ord_{s\,\in X(nN)}g(\tau) = A^2\ord_{Ms\,\in X(N)}f(\tau)$.
\item\label{gleadingcoeff} The leading coefficients of the Fourier expansion of $g$ around $s$ and the Fourier expansion of $f$ around $Ms$ have the same absolute value.
\end{enumerate}
\end{prop}
\begin{proof} \ref{gismodular} This can be proven by noting that $M\Gamma(nN)M^{-1}\subseteq\Gamma(N)$; for details, see the proof in \cite{L}.

\ref{matrixuppertri} Since $((\gamma')^{-1}M\gamma)\infty = (\gamma')^{-1}(Ms) = \infty$, we conclude that $(\gamma')^{-1}M\gamma$ is upper-triangular.

\ref{orderofg} Let $q_1(\tau)\coloneqq e^{2\pi i\gamma^{-1}\tau/N}$ and $q_2(\tau)\coloneqq e^{2\pi i(\gamma')^{-1}\tau/nN}$.  Then \begin{align*}\ord_{s\,\in X(nN)}g(\tau)&=\ord_{q_1(\tau)=0}f(M\tau)\\&=\ord_{q_1(M^{-1}\tau)=0}f(\tau)\\&=A^2\ord_{q_2(\tau)=0}f(\tau)\\&=A^2\ord_{Ms\,\in X(N)}f(\tau),\end{align*}
since $q_1(M^{-1}\tau)=e^{2\pi i\gamma^{-1}M^{-1}\tau/nN}=e^{2\pi i\begin{bsmallmatrix}\!A&B\!\\\!0&D\!\end{bsmallmatrix}^{-1}(\gamma')^{-1}\tau/nN}$ and hence $q_1(M^{-1}\tau)^{An/D}e^{2\pi iB/ND}=q_2(\tau)$.  But $An/D=A^2$ since $AD=\det (\gamma')^{-1}M\gamma=\det M=n$, as desired.

\ref{gleadingcoeff} Since $q_1(M^{-1}\tau)^{A^2}$ and $q_2(\tau)$ differ by an $ND$th root of unity, so do the leading coefficients of the Fourier expansions of $f$ around $Ms$ and $g$ around $s$.
\end{proof}

If we can write the cusp $s=\alpha/\beta$ where $\alpha$ and $\beta$ are relatively prime integers, we may compute that \begin{equation}\label{matrixfootnote}A = \gcd(a\alpha+b\beta, c\alpha+d\beta)\end{equation} where $M=\begin{bsmallmatrix}\!a&b\!\\\!c&d\!\end{bsmallmatrix}$; this fact will be very helpful for the computations in Section \ref{galconj}.

As mentioned in the introduction, the division values of the Weierstrass $\wp$-function
\[e^{(N)}_{\alpha,\beta}\coloneqq\wp\Big(\frac{\alpha\tau+\beta}{N};\tau\Big),\]
constitute fundamental examples of modular forms of weight $2$ for the group $\Gamma(N)$, where by $\wp(z;\tau)$ we refer to the Weierstrass $\wp$-function associated to the lattice $\langle1,\tau\rangle$.  The division values of the Weierstrass $\wp$-function satisfy the following transformation identity (which is stated by Monsky \cite{M1} in special cases for $N$), which follows from the definition of $\wp(z;\tau)$.
\begin{lemma}\label{divvalidentity} Let $N$ be a positive integer and let $(\alpha,\beta)\not\equiv(0,0)\!\pmod{N}$.  For any $\gamma=\begin{bsmallmatrix}\!a&b\!\\\!c&d\!\end{bsmallmatrix}\in\SL_2\Z$,
\[e^{(N)}_{\alpha,\beta}(\gamma\tau)=(c\tau+d)^2e^{(N)}_{a\alpha+c\beta,b\alpha+d\beta}(\tau).\]
Further, if $(\alpha,\beta)\equiv\pm(\alpha',\beta')\!\pmod{N}$, $e^{(N)}_{\alpha,\beta}(\tau) = e^{(N)}_{\alpha',\beta'}(\tau)$.
\end{lemma}
From the Fourier expansion of the Weierstrass $\wp$-function (see Silverman \cite{Silv2}, Proposition I.6.2), we can deduce that
\begin{equation}\frac{1}{(2\pi i)^2}e^{(N)}_{\alpha,\beta}(\tau)=\frac{1}{12}+\frac{e^{2\pi i\beta/N}q^\alpha}{(1-e^{2\pi i\beta/N}q^\alpha)^2}+F(e^{2\pi i\beta/N}q^\alpha, q^N)+F(e^{-2\pi i\beta/N}q^{-\alpha}, q^N) - 2F(1, q^N),\label{fourierexp}\end{equation}
where $q\coloneqq e^{2\pi i\tau/N}$ and
\[F(T,X)\coloneqq\sum_{n=1}^\infty\frac{TX^n}{(1-TX^n)^2}=TX+(2T^2+T)X^2+(3T^3+T)X^3+(4T^4+2T^2+T)X^4+\ldots\;.\]
We will now turn to finding modular functions $X$ and $Y$ of weight $0$ for $\Gamma(6)$ such that $Y^2 = X^3 + 1$.  Let $\omega\coloneqq e^{2\pi i/3}$, and define the modular forms\newpage
\begin{align*}
A(\tau) &\coloneqq e^{(6)}_{2,1}(\tau) + \omega^2e^{(6)}_{2,3}(\tau) + \omega e^{(6)}_{2,5}(\tau),\\
B(\tau) &\coloneqq e^{(6)}_{2,1}(\tau) + \omega e^{(6)}_{2,3}(\tau) + \omega^2e^{(6)}_{2,5}(\tau),\\
C(\tau) &\coloneqq e^{(6)}_{0,1}(\tau) - e^{(6)}_{0,4}(\tau),\\
D(\tau) &\coloneqq e^{(6)}_{3,1}(\tau) - e^{(6)}_{3,4}(\tau),
\end{align*}
of weight $2$ for $\Gamma(6)$.  It follows from Lemma \ref{divvalidentity} that the associated functions $A^3$, $B^3$, $C$, and $D^2$ are modular forms of weights $6$, $6$, $2$, and $4$ respectively for the group $\Gamma_0(6)$.  By Lemma \ref{inequivcusps}, the set $\{0,\infty,1/2,1/3\}$ forms a set of inequivalent cusps for $\Gamma_0(6)$; we will show that all zeros and poles of $A^3$, $B^3$, $C$, and $D^2$ lie in this set by computing the orders at the cusps and applying Proposition \ref{zeropolesum}.  Using the Fourier expansion of the Weierstrass $\wp$-function, we can prove the following result.

\begin{lemma} The table below provides the orders of the modular forms $A^3$, $B^3$, $C$, and $D^2$ at points of the complete set of inequivalent cusps $\{0,\infty, 1/2, 1/3\}$ for $\Gamma_0(6)$.  Moreover, the modular forms $A$, $B$, $C$, and $D$ have no zeros or poles on $\uhp$.\label{tableoforders}
\begin{table}[h]\centering
$\begin{array}{c|cccc}
        \ord_s f & \:\,0\:\,& \,\infty\, & 1/2 & 1/3\\
        \hline \\[-0.9em]
        A^3 &3 &1 &0 &2\\[0.3em]
        B^3 &3 &2 &0 &1\\[0.3em]
        C   &1 &0 &1 &0\\[0.3em]
        D^2 &2 &1 &1 &0
\end{array}$\end{table}
\end{lemma}
\begin{proof}For the first part of the lemma, we will prove that $\ord_{1/3}A^3=2$, the proofs for the other entries of the table are similar.  Let $q\coloneqq e^{\pi i\tau/3}$.  Now $1/3 = \gamma\infty$ where $\gamma=\begin{bsmallmatrix}\!1&0\!\\\!3&1\!\end{bsmallmatrix}$, hence
\[\ord_{1/3} A^3=\ord_\infty(A^3)_\gamma=\ord_\infty\big(e^{(6)}_{5,1}+\omega^2 e^{(6)}_{5,3}+\omega e^{(6)}_{5,5}\big)^3=\ord_\infty\big(\omega e^{(6)}_{1,1}+\omega^2 e^{(6)}_{1,3}+e^{(6)}_{1,5}\big)^3.\]
But, note that by equation (\ref{fourierexp}), $e^{(6)}_{1,5}(q)=e^{(6)}_{1,3}(\omega q)=e^{(6)}_{1,1}(\omega^2 q)$, so we see that the smallest non-zero term in the Fourier expansion of $\omega e^{(6)}_{1,1}+\omega^2 e^{(6)}_{1,3}+e^{(6)}_{1,5}$ has exponent $2$.  Since the local parameter around $1/3\in X_0(6)$ is $q_3(\tau)=q(\gamma^{-1}\tau)^3$, we deduce that $\ord_{1/3} A^3=2$, as desired.

For the second part of the lemma, since $A^3$, $B^3$, $C$, and $D^2$ are modular forms, they have no poles on $\uhp^*$.  To show that they have no zeros on $\uhp$, note by Theorem \ref{zeropolesum} that the sum of the orders of the zeros of a modular form of weight $k$ for $\Gamma_0(6)$ is \[\frac{k}{12}\,[\SL_2\Z:\{\pm1\}\cdot\Gamma_0(6)]=\frac{k}{12}\bigg[\:\!6\bigg(1+\frac{1}{2}\bigg)\!\bigg(1+\frac{1}{3}\bigg)\!\bigg]=k,\]
and that this value is already attained by the sum of the orders at the cusps $0$, $\infty$, $1/2$, and $1/3$ for each of the modular forms $A^3$, $B^3$, $C$, and $D^2$.  
\end{proof}
\begin{defn}Let $X(\tau)\coloneqq-\omega^2A(\tau)/B(\tau)$ and $Y(\tau)\coloneqq -3C(\tau)/D(\tau)$.  
\end{defn} 
We can then compute that the Fourier expansions around $\infty$ of $X(\tau)$ and $Y(\tau)$ are
\begin{align*}X(\tau) &= q^{-2} + q^4 + q^{10} - q^{16} - q^{22} + q^{34} + 2q^{40} + \ldots\intertext{and} Y(\tau) &= q^{-3} + 2q^3 + q^9 - 2q^{21} - 2q^{27} + 2q^{33} + \ldots,\end{align*}
where, as before, $q=e^{\pi i\tau/3}$.  (Note that all coefficients of both Fourier expansions are integers.)
\begin{thm} The functions $X^3$ and $Y^2$ are modular functions of weight $0$ for the group $\Gamma_0(6)$, and $Y^2 = X^3 + 1$.  The map $\phi\colon X(6)\to E_1$ such that $\phi(\tau)=(X(\tau), Y(\tau))$ for all $\tau\in X(6)\setminus\{\infty\}$ is an isomorphism.\label{themodularparam}\end{thm}
\begin{proof} The fact that $X^3$ and $Y^2$ are modular functions of weight $0$ for $\Gamma_0(6)$ and have no poles in $\uhp$ follows from Lemma \ref{tableoforders}.  From this lemma, we can also deduce that $X^3$ and $Y^2$ have the following orders at the cusps of $X_0(6)$:
\begin{table}[h]\centering
$\begin{array}{c|cccc}
        \ord_s f & \:\,0\:\,& \,\infty\, & 1/2 & 1/3 \\
        \hline \\[-0.9em]
        X^3 &0 &-1 &0 &1\\[0.3em]
        Y^2 &0 &-1 &1 &0
\end{array}$\end{table}

\noindent Based on the Fourier expansions for $X$ and $Y$ at the cusp $\infty$, we find that $\ord_\infty(Y^2-X^3)=0$.  Thus by Proposition \ref{zeropolesum}, Lemma \ref{tableoforders}, and the table above, the modular function $Y^2-X^3$ has nonnegative order at all points of $X_0(6)$, and by Proposition 2.3, must be identically constant. Inspecting the Fourier expansions of $X$ and $Y$ at $\infty$ one more time, we conclude that $Y^2-X^3 = 1$, as desired.

To see that $\phi$ is an isomorphism, note that as algebraic morphisms, $X\colon X(6)\to\mathbb{P}^1$ has degree $2$ and the coordinate function $x\colon E_1\to\mathbb{P}^1$ has degree $2$, hence $\phi\colon X(6)\to\mathbb{P}^1$ has degree $1$ and is an isomorphism. \end{proof}

We will determine the transformation properties of $X$ and $Y$ under specific elements of $\SL_2\Z$ which will appear in our later calculations (this result is an analogue of \cite{Sa}, Proposition 1.20).
\begin{prop} Let $\gamma_\pm, \gamma'\in\SL_2\Z$ such that $\gamma_\pm\equiv\id\!\pmod{2}$, $\gamma_\pm\equiv T^{\pm1}S\!\pmod{3}$, $\gamma'\equiv\id\!\pmod{3}$, and $\gamma'\equiv TS\!\pmod{2}$.  Then we have the identities\label{modparamidentities}
\begin{align*} \phi(\gamma_\pm\tau) &= [\omega^{\mp1}]\phi(\tau)+(-\omega^{\mp1}, 0),\\[0.05em]
\phi(\gamma'\tau) &= \phi(\tau)+(0,1).\end{align*}
\end{prop}
\begin{proof} Since $\Gamma(6)\unlhd\,\SL_2\Z$, every element $\gamma\in\SL_2\Z$ induces an isomorphism $X(6)\to X(6)$.  After composing with the isomorphism $\phi\colon X(6)\to E_1$ and its inverse, we see that $\gamma$ induces an isomorphism of varieties $f_\gamma\colon E_1\to E_1$.  This isomorphism of varieties can be written in the form $f_\gamma(Q)=[\alpha_\gamma]Q+P_\gamma$ for some $\alpha_\gamma\in\Z[\omega]$ and some $P_\gamma\in E_1(\overline{\Q})$.  Now $\gamma_\pm$ and $\gamma'$ clearly have order $3$ in $\SL_2\Z/(\pm\Gamma(6))$, which implies that $\alpha_\gamma\in\{1,\omega,\omega^2\}$ for all $\gamma\in\{\gamma_\pm,\gamma'\}$.  By equation (\ref{fourierexp}), we can compute the special values
\[\begin{array}{ccc}\phi(\infty)=O,&\phi(1/3)=(0,1),&\phi(1/2)=(-\omega, 0),\\[0.15em]\hspace{1.8888em}\phi(-3/2)=(-1,0),\quad&\hspace{2.8888em}\phi(-1)=(2\omega, -3),\quad\;&\hspace{0.9444em}\phi(-1/2)=(-\omega^2,0).\quad\;\end{array}\]
The proposition then follows by solving for the coefficients $\alpha_\gamma$ and $P_\gamma$ for all $\gamma\in\{\gamma_\pm, \gamma'\}$.\end{proof}
We also need to determine the behavior of the modular function $X(\tau)+1$ around the cusps, which will be useful in Section \ref{galconj}.  The following result can be deduced from the Fourier expansions of $X(\tau)+1$ at the cusps of $X(6)$, which we compute from equation (\ref{fourierexp}).
\begin{lemma}\label{xplusone} Suppose that $\gcd(\alpha,\beta)=1$.
\begin{enumerate}[label=\textup{(\alph*)}]
\item The modular function $X(\tau)+1$ has no zeros or poles in $\uhp$, and
\[\ord_{[\alpha/\beta]\in X(6)}(X(\tau)+1)=\begin{cases}\\[-1.7777em] -2 &\quad\text{if }6\mid \beta\\2 &\quad\text{if }3\mid\alpha,2\mid\beta\\0 &\quad\text{otherwise}\quad.\vspace{-0.2777em}\end{cases}\]
\item The leading coefficient of the Fourier expansion of $X(\tau)+1$ around the cusp $\alpha/\beta\in X(6)$ is
\[C=\begin{cases}\\[-1.7777em] 1&\quad\text{if }3\mid\beta \\3&\quad\text{if }3\mid\alpha\\[-0.05em]\mp\sqrt{-3}&\quad\text{if }\alpha\equiv\pm\beta\br\pmod{3}\,.\vspace{-0.2777em}\end{cases}\]
\end{enumerate}
\end{lemma}
We now turn to the Shimura Reciprocity Theorem (due to Shimura \cite{Shim}), the main application of which is to compute the fields of definition and Galois conjugates of values of modular functions at quadratic irrational numbers in $\uhp$.  Associated to the division values of the Weierstrass $\wp$-function are the modular functions of weight $0$ known as the \textit{Fricke functions}\footnote{Our definition is closest to that in \cite{M2} and that utilized by Shimura \cite{Shim}; the definition of Fricke functions in \cite{L} have an extra factor of $-31104$.}
\[f^{(N)}_{\alpha,\beta}(\tau)\coloneqq\frac{g_2(\tau)g_3(\tau)}{\Delta(\tau)}e^{(N)}_{\alpha,\beta}(\tau)\]
for any $N\in\N$ and $(\alpha,\beta)\not\equiv(0,0)\!\pmod{N}$.

Let $K$ be an imaginary quadratic field and let $z\in K\cap\uhp$.  Define the map $q_z\colon K^\times\to\GL_2\Q$ such that for any $w\in K^\times$, $q_z(w)$ is the unique matrix such that
\[q_z(w)\!\begin{bmatrix}z\\1\end{bmatrix} = w\!\begin{bmatrix}z\\1\end{bmatrix}\!.\]
This map induces a composite map \[q_{z*}\colon\mathbb{A}_K^\times\cong(\mathbb{A}_\Q\otimes_\Q K)^\times\to (\mathbb{A}_\Q\otimes_\Q \Mat_{2\times2}\Q)^\times\hookrightarrow\prod_{\nu\in M_\Q}\GL_2(\Q_\nu)\]
where $M_\Q$ is the set of places of $\Q$.  Then the Shimura Reciprocity theorem expresses the Galois conjugates of division values of the Weierstrass $\wp$-function in terms of the map $q_{z*}$. 
\begin{thm}\textup{\cite{Shim} (Shimura Reciprocity)}  In the setting above, $j(z), f^{(N)}_{\alpha,\beta}(z)\in K^{ab}$, and for any id\`ele $s\in \mathbb{A}_K^\times$,\vspace{-0.5em} 
\begin{align*}[s,K]j(z)&=j(\gamma z) \shortintertext{and} [s,K]f^{(N)}_{\alpha,\beta}(z) &= f^{(N)}_{(\alpha,\beta)U}(\gamma z)\end{align*}
for any $U\in\prod_{\nu\in M_\Q} \GL_2 \Z_\nu$ and $\gamma\in\GL_2\Q$ such that $\det\gamma>0$ and $U\gamma = q_{z*}(s)^{-1}$.\label{shimurareciprocity}
\end{thm}

In order to be able to apply the Shimura Reciprocity theorem more easily, we will also need the following important corollary (which is a more general version of the first part of Shimura \cite{Shim}, Proposition 6.34).

\begin{cor} Let $\mathscr{O}\subseteq K$ be an order in $K$ with conductor $f$.  Choose $z\in K\cap\uhp$ such that $\mathscr{O}$ is the lattice $\langle1,z\rangle$.  Then for any $N\in\N$ and any integers $\alpha,\beta$ with $(\alpha,\beta)\not\equiv(0,0)\!\pmod{N}$, the special value $f^{(N)}_{\alpha,\beta}(z)$ is defined over the ray class field of $K$ with modulus $fN$.\label{rayclassfield}
\end{cor}
\begin{proof} By the definition of ray class fields from class field theory, it suffices to show that $[s,K]$ acts trivially on $f^{(N)}_{\alpha,\beta}(z)$ for any idele $s\in\mathbb{A}_K^\times$ expressible as $s=u\gamma$, where $\gamma\in K^\times$, and $u\in\mathbb{A}_K^\times$ is congruent to $1\!\pmod{fN}$.  Under these assumptions on $\gamma$ and $u$, the map $[\gamma,K]$ is the identity automorphism and $U\coloneqq q_{z*}(u^{-1})$ is congruent to $\id\!\pmod{N}$, since $\mathscr{O}$ has conductor $f$ and hence $q_z(w)\equiv\id\!\pmod{N}$ whenever $w\equiv1\!\pmod{fN}$.  Then $(\alpha,\beta)U=(\alpha,\beta)$, so that $[s,K]f^{(N)}_{\alpha,\beta}(z)=f^{(N)}_{\alpha,\beta}(z)$, as desired.
\end{proof}\vspace{0.3em}
\section{\label{nonhol}Non-holomorphic Eisenstein Series.}
The non-holomorphic Eisenstein series (in the sense of Iwaniec \cite{I2}) are a class of functions on the upper half-plane which satisfy a periodicity relation with respect to some congruence subgroup $\Gamma\le\SL_2\Z$ as do modular functions.  These series are not holomorphic or even harmonic, but they are eigenfunctions of the Laplacian operator $\Delta\coloneqq y^2\big(\frac{\partial^2}{\partial x^2}+\frac{\partial^2}{\partial y^2}\big)$.  We will review the relevant general theory provided by Iwaniec \cite{I2} and make calculations in specific cases (Propositions \ref{phiisholomorphic} and \ref{eisenstein}).  As the material from this section will only be used in Lemma \ref{Fisharmonic} and Theorem \ref{harmonicfns}, a reader who wishes to jump to the proof of Theorem \ref{fundunitcube} may skip this section and refer back to it as needed.

The non-holomorphic Eisenstein series associated to $\Gamma(N)$ are defined in \cite{I2} as follows.
\begin{defn} For any cusp $t\in X(N)$, and any matrix $\sigma\in\SL_2\Z$ such that $t=\sigma\infty$, the \textit{non-holomorphic Eisenstein series} at the cusp $t\in X(N)$ is
\[E_{t\in X(N)}(\tau,s)=\sum_{\gamma\in\langle T^N\rangle\backslash\Gamma(N)}(N\Im\sigma^{-1}\gamma\tau)^s,\]
whenever $\Re s>1$.
\end{defn}
\noindent To address our problem, we will need to determine the Fourier expansions of these Eisenstein series and analyze them in the limit $s\to1$.  These expansions can be written in terms of \textit{Kloosterman sums} \cite{I2}.
\begin{defn} Let $c$, $n$ and $N$ be positive integers, and let $t,u\in\uhp^*$ be cusps for $X(N)$. Choose $\gamma_t, \gamma_u\in\SL_2\Z$ such that $t=\gamma_t\infty$ and $u=\gamma_u\infty$.  Then the sum
\[S_{tu}(0,n;c)\coloneqq\sum_{\gamma=\begin{bsmallmatrix}\!a&b\!\\\!c&d\!\end{bsmallmatrix}\in\langle T\rangle\backslash a(1/N)\gamma_t^{-1}(\pm\Gamma(N))\gamma_ua(N)/\langle T\rangle} e^{2\pi ina/c}\]
is known as a \textit{Kloosterman sum}, where
\[a(x)\coloneqq\begin{bmatrix}\sqrt{x}&0\\0&1/\sqrt{x}\end{bmatrix}.\]
\end{defn}
The following result, which is Theorem 3.4 of \cite{I2}, details the Fourier expansion around the cusp $u$ of the non-holomorphic Eisenstein series associated to the cusp $t\in X(N)$.
\begin{thm}\label{nhfourierexp} \textup{\cite{I2}} If $\Re s>1$, $t$ and $u$ are cusps of $X(N)$, and $\gamma_u\in\SL_2\Z$ such that $\gamma_u\infty=u$,
\[E_{t\in X(N)}(\gamma_u\tau,s)=\delta_{tu}N^{-s}(\Im\tau)^s+\varphi_{tu}(s)N^{1-s}(\Im\tau)^{1-s}+\sum_{n\ne0}\varphi_{tu}(n,s)W_s(n\tau/N),\]
where
\begin{align*}\varphi_{tu}(s)&\coloneqq\sqrt{\pi}\frac{\Gamma(s-1/2)}{\Gamma(s)}\sum_{c=1}^\infty \frac{1}{c^{2s}}S_{tu}(0,0;c),\\\varphi_{tu}(n,s)&\coloneqq\frac{\pi^s}{\Gamma(s)}\lvert n\rvert^{s-1}\sum_{c=1}^\infty \frac{1}{c^{2s}}S_{tu}(0,n;c),\\W_s(z)&\coloneqq2\abs{\Im z}^{1/2}e^{2\pi i\Re z}K_{s-1/2}(2\pi\abs{\Im z}),\end{align*}
and $\delta_{tu}=1$ if $t$ and $u$ are equivalent cusps of $X(N)$ and $\delta_{tu}=0$ otherwise.  Here $K_s(x)$ denotes the $K$-Bessel function (see Section B.4 of \cite{I2}).
\end{thm}
Our aim in this section is to rewrite these Fourier expansions in a way so that they can be analytically continued past the line $\Re s=1$.  First, let us simplify the Kloosterman sum $S_{tu}(0,n;c)$.  In the notation above, suppose that
\[\gamma_u^{-1}\gamma_t=\begin{bmatrix}\alpha&c_1\\\beta&c_2\end{bmatrix}.\]
Since $\Gamma(N)\unlhd\SL_2\Z$, we see that
\[a(1/N)\gamma_t^{-1}(\pm\Gamma(N))\gamma_ua(N)=\bigg\{\!\begin{bmatrix}A&B/N\\CN&D\end{bmatrix}\in\SL_2\Q:\begin{bmatrix}A&B\\C&D\end{bmatrix}\equiv\pm\begin{bmatrix}c_2&-c_1\\-\beta&\alpha\end{bmatrix}\br\pmod{N}\!\bigg\}.\]
\begin{lemma} Let $\lambda\coloneqq\gcd(N,\beta)$.  Then $S_{tu}(0,0;c)=\delta_{N\mid c}(\delta_{c/N\equiv\beta\!\pmod{N}}+\delta_{c/N\equiv-\beta\!\pmod{N}})\frac{\lambda}{\phi(\lambda)}\phi\big(\!\frac{c}{N}\!\big)$.
\end{lemma}
\begin{proof}Note that $S_{tu}(0,0;c)$ is the number of residues $a\!\pmod{c}$ such that $\gcd(a,c/N)=1$ and $a\equiv\mp c_2\pmod{N}$, whenever $N\mid c$ and $c/N\equiv\pm\beta\!\pmod{N}$; otherwise, $S_{tu}(0,0;c)=0$. Indeed, for any $a,c\in\Z$ satisfying the first two conditions, there exists a unique $d\!\pmod{c}$ such that $\begin{bsmallmatrix}\!a&b\!\\\!c&d\!\end{bsmallmatrix}\in a(1/N)\gamma_t^{-1}(\pm\Gamma(N))\gamma_u a(N)$ for some $b\in N^{-1}\Z$; in fact, $b$ is uniquely determined.  Now $\lambda=\gcd(N, c/N)$, so if $N\mid c$ and $c/N\equiv\pm\beta\!\pmod{N}$, there are $\phi(c/N)/\phi(\lambda)$ residues modulo $c/N$ which are relatively prime to $c/N$ and are congruent to $\mp c_2\pmod{\lambda}$, and hence there are $\lambda\phi(c/N)/\phi(\lambda)$ residues modulo $c$ which are relatively prime to $c/N$ and are congruent to $\mp c_2\pmod{N}$, since $\lcm(c/N, N)=c/\lambda$.
\end{proof}
\begin{prop}\label{phiisholomorphic} For any nonzero integer $n$, $\abs*{S_{tu}(0,n;c)}\le2\sigma(n)$, and $\varphi_{tu}(n,s)$ is holomorphic in the domain $\Re{s}>1/2$.
\end{prop}
\begin{proof} For any integer $m$ such that $\gcd(m, c/N)=1$, define
\[S(m)\coloneqq\sum_{\substack{1\,\le\,a\,\le\,c\\a\equiv m\br\pmod{N}\\\gcd(a,\,c/N)=1}}e^{2\pi ian/c};\]
then
\[S_{tu}(0,n;c)=\delta_{N\mid c}(\delta_{c/N\equiv\beta\br\pmod{N}}S(-c_2)+\delta_{c/N\equiv-\beta\br\pmod{N}}S(c_2)).\]
Assume $N\mid c$, and let $w\coloneqq\gcd(m,N)$, $\eta\coloneqq\gcd(n,c/w)$, $k$ be an integer such that $\gcd(k,c/w)=1$ and $k\equiv n/\eta\!\pmod{c/w\eta}$, and $a_0$ be an integer such that $a_0\equiv km/w\!\pmod{N/w}$ and $((c/N)/\gcd(c/N, N/w))\mid a_0$.  In a similar manner to the proof of equation (2.26) in \cite{I2}, we find that
\begin{align*}S(m)&=\sum_{\substack{1\,\le\,a\,\le\,c/w\\a\equiv km/w\br\pmod{N/w}\\\gcd(a,\,c/N)=1}}e^{2\pi ia\eta/(c/w)}=\sum_{T\mid c/N}\mu(T)\sum_{\substack{1\,\le\,a\,\le\,c/w\\a\equiv km/w\br\pmod{N/w}\\T\mid a}}e^{2\pi ia\eta/(c/w)}\\&=\sum_{T\mid c/N}\mu(T)\delta_{\gcd(T,\,N/w)=1}\delta_{c\mid\eta TN}\frac{c}{TN}e^{2\pi ia_0\eta/(c/w)}\\&=\sum_{t\mid \gcd(\eta,\,c/N)}\mu\Big(\!\frac{c}{Nt}\!\Big)\delta_{\gcd(c/Nt,\,N/w)=1}te^{2\pi ia_0\eta/(c/w)}.\end{align*}
Since $\gcd(\eta,c/N)\le n$, we see that $\abs*{S(m)}\le\sigma(n)$ and $\abs*{S_{tu}(0,n;c)}\le2\sigma(n)$, as desired.  Since this estimate is independent of $c$, we conclude that $\varphi_{tu}(n,s)$ has a holomorphic continuation to the domain $\Re s>1/2$.
\end{proof}
\noindent Now, since $\lambda=\gcd(N,\beta)$,
\[\sum_{c\equiv\beta\br\pmod{N}}\frac{\phi(c)}{c^{2s}}=\frac{1}{\phi(N/\lambda)}\sum_{\substack{\chi\colon(\Z/(N/\lambda)\Z)^\times\to\C^\times\\\text{character}}}\frac{1}{\chi(\beta/\lambda)}\frac{\phi(\lambda)}{\lambda^{2s}}\frac{L(2s-1,\chi)}{L(2s,\chi)}\prod_{p\mid\lambda}\big(1-\chi(p)p^{-2s}\big)^{-1}.\]
Let $C\colon(\Z/N\Z)^2\to\C$ be a function invariant under the multiplication-by-$k$ map for all integers $k$ relatively prime to $N$.  By summing over the sets $\{(ka,kb)\in(\Z/N\Z)^2:k\in(\Z/N\Z)^\times\}$, we can prove the following.\newpage
\begin{prop}\label{eisenstein} As $\Im\tau\to\infty$,
\begin{align*}\sum_{\substack{1\,\le\,a,\,b\,\le\,N\\\gcd(a,\,b,\,N)=1}}C(a,b)E_{[a/b]\in X(N)}(\gamma_u\tau,s)&=\sum_{\substack{1\,\le\,a,\,b\,\le\,N\\\gcd(a,\,b,\,N)=1}}C(a,b)\delta_{[a/b],u}N^{-s}(\Im\tau)^s\\&+2\sqrt{\pi}\frac{\Gamma(s-1/2)}{\Gamma(s)}\frac{\zeta(2s-1)}{\zeta(2s)}\sum_{\substack{1\,\le\,a,\,b\,\le\,N\\\gcd(a,\,b,\,N)=1}}C(a,b)\frac{\lambda^{1-2s}N^{1-3s}}{\phi(N/\lambda)}\\&\qquad\times\raisebox{-0.15em}{$\bigg($}\prod_{p\mid N}\big(1-p^{-2s}\big)^{-1}\raisebox{-0.15em}{$\bigg)\!\bigg($}\prod_{p\mid N/\lambda}\big(1-p^{1-2s}\big)\!\raisebox{-0.15em}{$\bigg)$}(\Im\tau)^{1-s}\\&+O(e^{-2\pi\Im\tau}),\end{align*}
whenever $\Re s>1/2$.
\end{prop}\vspace{0.3em}
\section{\label{galconj}Proof of Theorem \ref{fundunitcube}.}
In this section, we will prove Theorem \ref{fundunitcube}. Suppose that $n$ is a positive integer relatively prime to $6$.  Let $\omega\coloneqq e^{2\pi i/3}$ (as before), $k\coloneqq \Q(\omega)$, and $K\coloneqq \Q(\!\sqrt[3]{n})$, and for any integer $m$, let $R_m$ denote the ring class field of the order $\Z[m\omega]\subseteq\Z[\omega]$.  Let $n=\prod_i\lambda_i^{c_i}$ be the prime factorization of $n$ in $\ok=\Z[\omega]$.  Our approach is as follows.  First, in order to understand the Galois conjugates of $X(n\omega)$, we will use the Shimura Reciprocity theorem to express them as other special values of $X$.  We will then rewrite the norm of $X(n\omega)+1$ over $R_{6n}/kK$ as the value at $\omega$ of a modular function $U$ for the group $\Gamma(6n)$. Next, we will express the logarithm of the absolute value of this modular function as the limit as $s\to1$ of an appropriate weighted sum of non-holomorphic Eisenstein series (Theorem \ref{harmonicfns}), and rewrite the weighted sum in terms of a Dirichlet $L$-function of $k$ with cubic character (Theorem \ref{lfunction}).  Finally, we will apply the class number formula to write the residue of this Dirichlet $L$-function at $s=1$ (which is related to the residue of the Dedekind zeta function of $K$ at $s=1$), and hence the norm of $X(n\omega)+1$, in terms of the fundamental unit of $K$.

For any $x=a+b\:\!\omega\in k^\times$, the unique matrix $M(x)\in\GL_2\Q$ such that $M(x)\begin{bsmallmatrix}\!n\omega\!\\\!1\!\end{bsmallmatrix}=x\begin{bsmallmatrix}\!n\omega\!\\\!1\!\end{bsmallmatrix}$ is
\[M(x)\coloneqq\begin{bmatrix}a-b&-bn\\b/n&a\end{bmatrix}\!;\]
this matrix $M(x)$ has determinant $Nx=a^2-ab+b^2$.  We now prove a result which details our particular application of the Shimura Reciprocity theorem and shows that $X(n\omega)\in R_{6n}$, which is a more general application than in \cite{Sa}.
\begin{prop}\label{MandS} Suppose that $x\in\Z[\omega]$ is relatively prime to $6n$, and that $S\in\SL_2\Q$ is a matrix such that $nS$ and $M(x)S^{-1}$ are integral, and $M(x)S^{-1}\equiv\id\pmod{6}$.  Then for any $(\alpha,\beta)\not\equiv(0,0)\!\pmod{6}$,
\[(x,\,K_{6n}/k)f^{(6)}_{\alpha,\beta}(n\omega)=f^{(6)}_{\alpha,\beta}(Sn\omega),\]
where $(-,\,K_m/k)$ is the Artin map.
\end{prop}
\begin{proof} We follow the general proof method in \cite{Sa}, but generalize many steps.  Suppose that $x$ factorizes in $\Z[\omega]$ into primes as $\pi_1^{e^{\vphantom{\prime}}_1}\overline{\pi}_1^{e'_1}\cdots\pi_r^{e^{\vphantom{\prime}}_r}\overline{\pi}_r^{e'_r}q_1^{f_1}\cdots q_\lambda^{f_\lambda}$, where $\pi_i\in\Z[\omega]$ is a prime of degree $1$ for each $1\le i\le r$ and $q_j\in\Z$ for each $1\le j\le \lambda$ is a prime of degree $2$ in $\Z[\omega]$.  Define the idele $s\in\mathbb{A}_k^\times$ by
\[s_\nu=\begin{cases}\\[-1.7777em]\pi_i^{e^{\vphantom{prime}}_i}\overline{\pi}_i^{e'_i}&\quad\text{if }\nu=\pi_i\text{ or }\overline{\pi}_i\\q_j^{f_j}&\quad\text{if }\nu=q_j\\1&\quad\text{otherwise}\hspace{2.1111em}.\vspace{-0.2777em}\end{cases}\]
Then $(x,k^{ab}/k)=[s,k]$, since $\pi$ is a unit in the ring of integers of the local field $k_{\overline{\pi}}$ for all primes $\pi\in\ok$ of degree $1$.  In the notation of the Shimura reciprocity theorem, $M(y) = q_{n\omega}(y)$ for any $y\in\Z[\omega]$, so $q_{n\omega*}(s)$ is the element of $\prod_{\nu\in M_\Q}\GL_2\Q_\nu$ with $q_{n\omega*}(s)_{\nu}=M(s_{\nu'})$ for all places $\nu$ of $\Q$ and $\nu'\mid\nu$.  If we define $\gamma\coloneqq SM(x)^{-1}$ and choose $U\in\prod_{\nu\in M_\Q}\GL_2\Z_\nu$ such that
\[U_\nu=\begin{cases}\\[-1.7777em]M\big(\pi_i^{e_i}\overline{\pi}_i^{e'_i}\big)^{-1}M(x)S^{-1}&\quad\text{if }\nu=N\pi_i\\M\big(q_j^{f_j}\big)^{-1}M(x)S^{-1}&\quad\text{if }\nu=q_j\\M(x)S^{-1}&\quad\text{otherwise}\:\:,\vspace{-0.2777em}\end{cases}\]
we see that $U\gamma=q_{n\omega*}(s)$.  Note that $U_\nu$ is invertible over $\Z_\nu$ because $\det M(\pi_i^{e_i}\overline{\pi}_i^{e'_{\smash{i}}})^{-1}M(x)$ is relatively prime to $N\pi_i$, $\det M(q_j^{f_j})^{-1}M(x)$ is relatively prime to $q_j$, and $\det M(x)S^{-1}=Nx$ is relatively prime to all primes not dividing $Nx$.  Since $U_2=U_3=M(x)S^{-1}\equiv\id\!\pmod{6}$, we see that $f^{(6)}_{(\alpha,\beta)U}=f^{(6)}_{\alpha, \beta}$.  Since $M(x)^{-1}(n\omega)=n\omega$, we conclude by the Shimura Reciprocity theorem that
\[(x,\,K_{6n}/k)f^{(6)}_{\alpha,\beta}(n\omega)=[s,k]f^{(6)}_{\alpha,\beta}(n\omega)=f^{(6)}_{\alpha,\beta}(\gamma n\omega)=f^{(6)}_{\alpha,\beta}(Sn\omega),\]
as desired.\end{proof}

We can choose an $S$ satisfying the conditions of Proposition \ref{MandS} as follows.  Suppose $x=a+b\:\!\omega$, and let $q=\pm\gcd(a,b)$, the sign chosen such that $q\equiv1\!\pmod{6}$.  Let $r,s\in\Z$ such that $r(a/q)+s(b/q)=((Nx)/q^2-1)/6$.  Then
\[S(x)\coloneqq\begin{bmatrix}(a-b)/q-6r&-(b/q-6s)n\\b/qn&a/q\end{bmatrix}\]
satisfies the desired conditions -- that $nS(x)$ and $M(x)S(x)^{-1}$ are integral, and $M(x)S(x)^{-1}\equiv\id\!\pmod{6}$.

By class field theory, for any positive integer $m$, the ring class field $R_m$ is the fixed field in the ray class field $K_m$ of $k$ with modulus $m$ of the image of $P_\Z(m)/P(m)$ under the Artin map $(-, K_m/k)$, where $P(m)$ is the group generated by principal fractional ideals of elements of $k$ congruent to $1\!\pmod{m}$, and $P_\Z(m)$ is the group generated by the principal ideals of elements of $\Z[m\omega]$ relatively prime to $m$ (see e.g.\ Cox \cite{Cox}, \S7--9).  An element $x\in K_m$ is then contained in $R_m$ if and only if $(\alpha,K_m/k)x=x$ for all $\alpha\in\Z[m\omega]$ relatively prime to $m$.  Since $S\equiv M(x)\equiv\id\!\pmod{6}$ whenever $x\in\Z[6n\omega]$, and $S$ is an integral matrix of determinant $1$, we deduce the following corollary.
\begin{cor} For any positive integer $n$ relatively prime to $6$, $\phi(n\omega)\in E_1(R_{6n})$.\end{cor}
We need to determine the set of Galois conjugates of $X(n\omega)$, which by Proposition \ref{MandS} reduces to finding a set of values $x\in\Z[\omega]$ such that the automorphisms $(x,R_{6n}/k)$ completely describe the group $\Gal(R_{6n}/K)$.  First, we prove that $K\subseteq R_{3n}$ and determine the action of Frobenius elements on $\sqrt[3]{n}$.
\begin{lemma}\label{froboncbrt}Suppose as above that $n$ is a positive integer relatively prime to $6$ and that $n=\prod_i \lambda_i^{c_i}$ is the prime factorization of $n$ in $\ok$.\begin{enumerate}[label=\textup{(\alph*)}]
\item\label{froboncbrtn} If $x\in\ok$ and $x\equiv\pm\omega^j\!\pmod{3}$ for some integer $j$, \[(x,\,k^{ab}/k)\sqrt[3]{n}=\omega^{-j(n^2-1)/3}\sqrt[3]{n}\,\prod_i\Big(\!\frac{x}{\lambda_i}\!\Big)_{\!3}^{\!c_i},\vspace{-0.9em}\]
and if $x\equiv\omega^r\!\pmod{2}$, $(x,\,k^{ab}/k)\sqrt[3]{2}=\omega^{r-j}\sqrt[3]{2}$.
\item\label{ringclassfields} $R_{6n}=R_{3n}(\!\sqrt[3]{2})$ and $\sqrt[3]{n}\in R_{3n}$.  In addition, $\sqrt[3]{n}\in R_n$ if and only if $n\equiv\pm1\!\pmod{9}$.
\end{enumerate}
\end{lemma}
\begin{proof} \ref{froboncbrtn} For any prime $\pi\in\Z[\omega]$ such that $\pi\equiv\pm\omega^j\!\pmod{3}$, it follows from cubic reciprocity that
\[\Frob_{\pi}\!\sqrt[3]{n}=\Big(\!\frac{n}{\pi\omega^{-j}}\!\Big)_{\!3}\sqrt[3]{n}=\sqrt[3]{n}\,\prod_i\Big(\!\frac{\lambda_i}{\pi\omega^{-j}}\!\Big)_{\!3}^{\!c_i}=\sqrt[3]{n}\,\prod_i\Big(\!\frac{\pi}{\lambda_i}\!\Big)_{\!3}^{\!c_i}\omega^{-jc_i(N\lambda_i-1)/3}.\]
But we can easily check, by casework on $n\!\pmod{9}$, that $\sum_i c_i(N\lambda_i-1)\equiv n^2-1\pmod{9}$, as desired.  The proof for the action of Artin isomorphisms on $\sqrt[3]{2}$ is similar.

\ref{ringclassfields} By part \ref{froboncbrtn}, $(x,k^{ab}/k)\sqrt[3]{2}$ only depends on the residue class $x\!\pmod{6}$, but does depend on the residue class $x\!\pmod{2}$, hence $R_{6n}=R_{3n}(\!\sqrt[3]{2})$ (since $[R_{6n}:R_{3n}]=3$).  Similarly, $(x,\,k^{ab}/k)\sqrt[3]{n}$ only depends on the class $x\!\pmod{3n}$, and only depends on the class $x\!\pmod{n}$ if and only if $\omega^{(n^2-1)/3}=1$, i.e., when $n\equiv\pm1\!\pmod{9}$.  By the computation of $\oK$ in \cite{Cas}, $kK/k$ is unramified at the prime $(\sqrt{-3})$ if and only if $n\equiv\pm1\!\pmod{9}$, so that $\sqrt[3]{n}\in R_n$ if and only if $n\equiv\pm1\!\pmod{9}$, as desired.
\end{proof}
\begin{prop} Let $\Lambda$ be a fundamental parallelogram for the lattice $6n\ok$ and let $A_i$ be a set of representatives in $\Lambda$ for the quotient $(\ok/n\ok)^\times/(\Z/n\Z)^\times$, all of which are congruent to $\omega^i\!\pmod{2}$ and to $1\!\pmod{3}$.  For any $i\in\{0,1,2\}$, let us define $B_i\coloneqq\big\{x\in A_i:\prod_i\!\big(\!\frac{x}{\lambda_i}\!\big)_3^{c_i}=1\big\}$.  Then the map $B_0\cup B_1\cup B_2\to\Gal(R_{6n}/kK)$ under which $x\mapsto(x,R_{6n}/k)$ is a bijection.
\end{prop}
\begin{proof} There exists an exact sequence
\[1\to(\Z/6n\Z)^\times\times\ok/\{\pm1\}\to(\ok/6n\ok)^\times\to I(6n)/P_\Z(6n)\cong\Cl\Z[6n\omega]\cong\Gal(R_{6n}/k)\to1,\]
where $I(6n)$ is the group of fractional ideals of $\ok$ relatively prime to $6n$ and $P_\Z(6n)$ is the group generated by the principal ideals of elements of $\Z[6n\omega]$ relatively prime to $6n$ (see Cox \cite{Cox}, \S7).  By the Chinese Remainder Theorem, the Artin map induces an isomorphism
\[\frac{(\ok/2\ok)^\times}{(\Z/2\Z)^\times}\times\frac{(\ok/3\ok)^\times}{(\Z/3\Z)^\times\times\{1,\omega,\omega^2\}}\times\frac{(\ok/n\ok)^\times}{(\Z/n\Z)^\times}\xrightarrow{\quad\raisebox{-0.2em}{\scriptsize$\cong$}\quad}\Gal(R_{6n}/k).\]
Since $A_0\cup A_1\cup A_2$ is a set of representatives for the elements of the group on the left-hand-side of the isomorphism, the Artin map induces a bijection $A_0\cup A_1\cup A_2\to\Gal(R_{6n}/k)$.  By the computation of the action of the Artin isomorphisms on $\sqrt[3]{n}$ in Lemma \ref{froboncbrt}\ref{froboncbrtn}, we deduce that this bijection restricts to a bijection $B_0\cup B_1\cup B_2\to\Gal(R_{6n}/kK)$, as desired.
\end{proof}
We can now write \[N_{R_{6D}/kK}(X(n\omega)+1)=U(\omega)\] where $U$ is the modular function of weight $0$ for the group $\Gamma(6n)$ defined as follows.
\begin{defn} For any $x\in B_0\cup B_1\cup B_2$, we define $R(x)\coloneqq S(x)\begin{bsmallmatrix}\!n&0\!\\\!0&1\!\end{bsmallmatrix}$, $U_x(\tau)\coloneqq X(R(x)\tau)+1$, and
\[U(\tau)\coloneqq\prod_{x\in B_0\cup B_1\cup B_2} U_x(\tau).\]\end{defn}
Our next goal in the proof of Theorem \ref{fundunitcube} is to determine the zeros and poles of $U(\tau)$, which we will use to prove that $\log{\abs*{U(\omega)}}$ can be expressed as a weighted sum of non-holomorphic Eisenstein series.
\begin{lemma}\label{ordersofconjugates} Let $x=a+b\:\!\omega\in B_i$.  Suppose that $\gcd(\alpha,\beta)=1$, and define $\lambda\coloneqq\gcd(n,b\alpha+a\beta)$.  Then
\[\ord_{[\alpha/\beta]\in X(6n)}U_x(\tau)=\begin{cases}\\[-1.7777em]-2\lambda^2&\quad\text{if }3\mid\beta,\,\alpha-\beta\omega\equiv\omega^i\br\pmod{2}\\2\lambda^2&\quad\text{if }3\mid\alpha,\,\alpha-\beta\omega\equiv\omega^i\br\pmod{2}\\0&\quad\text{otherwise}\hspace{8.5em}.\vspace{-0.2777em}\end{cases}\]
\end{lemma}
\begin{proof} Let us write $S(x)=\begin{bsmallmatrix}\!u&vn\!\\\!w/n&z\!\end{bsmallmatrix}$; then $u,v,w$, and $z$ are integers such that $uz-vw=1$, $w=b/q$, and $z=a/q$.  Since $\gcd(un\alpha+vn\beta, w\alpha+z\beta)\mid n\gcd(u\alpha+v\beta, w\alpha+z\beta)=n\gcd(\alpha,\beta)=n$, we find that
\[\gcd(un\alpha+vn\beta, w\alpha+z\beta)=\gcd(n,w\alpha+z\beta)=\gcd(n,b\alpha+a\beta)=\lambda.\]
By Lemma \ref{matrixtransform}\ref{orderofg} and equation (\ref{matrixfootnote}), we deduce that
\[\ord_{[\alpha/\beta]\in X(6n)}U_x(\tau)=\lambda^2\ord_{[R(x)(\alpha/\beta)]\in X(6)}(X(\tau)+1).\]
By the definition of $S(x)$, we see that if $x\in B_0$, $[R(x)(\alpha/\beta)]\in X(6)$ is equivalent to $[n\alpha/\beta]$; if $x\in B_1$, $[R(x)(\alpha/\beta)]\in X(6)$ is equivalent to $[n(\alpha-3\beta)/(3\alpha+4\beta)]$; and if $x\in B_2$, $[R(x)(\alpha/\beta)]\in X(6)$ is equivalent to $[n(-2\alpha-3\beta)/(3\alpha+\beta)]$.  The lemma then follows from the fact that $X(\tau)+1$ has a double pole at $\infty\in X(6)$, a double zero at $[3/2]\in X(6)$, and no other zeros or poles in $X(6)$.
\end{proof}
\noindent To compute the orders of $U(\tau)$ at the zeros and poles, we need to make the following definition.
\begin{defn}For any relatively prime integers $\alpha$ and $\beta$, and any divisor $d\mid n$, choose $i\in\{0,1,2\}$ such that $\alpha-\beta\omega\equiv\omega^i\!\pmod{2}$.  We then define the three quantities
\begin{align*} c_{\alpha,\beta}(d)&\coloneqq\abs*{\{x=a+b\:\!\omega\in B_i: d\mid a\beta+b\alpha\}},\\f(d)&\coloneqq\abs*{\frac{(\ok/d\ok)^\times}{(\Z/d\Z)^\times}}=d\prod_{p\mid d}\Big[1-p^{-1}\Big(\frac{p}{3}\Big)\Big],\\\delta&\coloneqq\gcd(n, \alpha^2+\alpha\beta+\beta^2).
\end{align*}
\end{defn}
\begin{lemma}\label{divisorcoeff} Suppose that $\gcd(\alpha,\beta)=1$. Let $n'$ be the largest squarefree divisor of $n$, and suppose that $3\nmid\nu_p(n)$ for all prime numbers $p\mid n$.  Let $d$ be any divisor of $n$.
\begin{enumerate}[label=\textup{(\alph*)}]
\item If $n'\mid d$,
\[c_{\alpha,\beta}(d)=\begin{cases}\\[-1.7777em]f(n)/f(d)&\quad\text{if }\prod_i\big(\!\frac{\alpha-\beta\omega}{\lambda_i}\!\big)_3^{c_i}=1\\0&\quad\text{otherwise}\hspace{4.1667em}.\vspace{-0.2777em}\end{cases}\]
\item If $n'\nmid d$,
\[c_{\alpha,\beta}(d)=\begin{cases}\\[-1.7777em]f(n)/3f(d)&\quad\text{if }\gcd(d,\delta)=1\\0&\quad\text{otherwise}\hspace{2.6667em}.\vspace{-0.2777em}\end{cases}\]
\end{enumerate}
\end{lemma}
\begin{proof} Since $\gcd(\alpha,\beta)=1$, $d\mid a\beta+b\alpha$ if and only if $(a,b)\equiv c(\alpha,-\beta)\!\pmod{d}$ for some integer $c$ relatively prime to $d$.  Now note that there are $f(n)/f(d)$ residue classes $x\!\pmod{n}$ (for $x\in\ok$) that are congruent to $\alpha-\beta\omega\!\pmod{d}$, and hence there are $f(n)/f(d)$ elements of $A_i$ congruent to $c(\alpha-\beta\omega)\!\pmod{d}$ for some integer $c$ relatively prime to $d$.  If $n'\mid d$, any $x\in A_i$ such that $x\equiv c(\alpha-\beta\omega)\!\pmod{d}$ is an element of $B_i$ if and only if $\prod_i\big(\!\frac{x}{\lambda_i}\!\big)_3^{c_i}=\prod_i\big(\!\frac{\alpha-\beta\omega}{\lambda_i}\!\big)_3^{c_i}=1$; hence under this condition, $c_{\alpha,\beta}(d)=f(n)/f(d)$, else $c_{\alpha,\beta}(d)=0$.  If $n'\nmid d$ and $\gcd(d,\delta)=1$, we see that for any $\lambda_j\mid n$ that does not also divide $d$, $\big(\!\frac{x}{\lambda_j}\!\big)_3^{c_j}=1$, $\omega$, or $\omega^2$, with each value occurring for exactly $1/3$ of the residues modulo $N\lambda_j$; hence $c_{\alpha,\beta}(d)=f(n)/3f(d)$.  If $\gcd(d,\delta)\ne 1$, there exists some $\lambda_j$ dividing both $d$ and $\alpha-\beta\omega$.  Because $\lambda_j\nmid x$ for all $x\in B_i$, we conclude that $c_{\alpha,\beta}(d)=0$.
\end{proof}
\noindent Next, we define the expression
\[C(\alpha,\beta)\coloneqq\!\sum_{\substack{d\mid n\\\gcd(d,\delta)=1}}\br d^2\frac{f(n)}{3f(d)}\prod_{p\mid d}\big(1-p^{-2}\big)+\frac{2}{3}\Re\prod_{i}\bigg(\!\frac{\alpha-\beta\omega}{\lambda_i}\!\bigg)_{\!3}^{\!c_i}\sum_{n'\mid d\mid n}dn\prod_{p\mid d}\big(1-p^{-2}\big)\,.\]
\begin{prop} Let $n'$ be the largest squarefree divisor of $n$, and suppose that $3\nmid\nu_p(n)$ for all prime numbers $p\mid n$.  If $\gcd(\alpha,\beta)=1$, then
\[\ord_{[\alpha/\beta]\in X(6n)}U(\tau)=2wC(\alpha,\beta),\]
where $w=1$ if $3\mid\alpha$, $w=-1$ if $3\mid\beta$, and $w=0$ otherwise.
\end{prop}
\begin{proof} Note that if $i\in\{0,1,2\}$ such that $\alpha-\beta\omega\equiv\omega^i\pmod{2}$, by Lemma \ref{ordersofconjugates},
\begin{align*}\ord_{[\alpha/\beta]\in X(6n)}U(\tau)&=\sum_{x\in B_i}\ord_{[\alpha/\beta]\in X(6n)}U_x(\tau)\\&=2w\sum_{d\mid n}d^2\:\!\abs*{\{x=a+b\:\!\omega\in B_i:d=\gcd(n,a\beta+b\alpha)\}}\\&=2w\sum_{d\mid n}d^2c_{\alpha,\beta}(d)\prod_{p\mid d}\big(1-p^{-2}\big),\end{align*}
so the result follows from Lemma \ref{divisorcoeff} (since $f(n)/f(d)=n/d$ when $n'\mid d$).\end{proof}
We can now use our computation of the orders of $U(\tau)$ at the cusps of $X(6n)$ to define an appropriate weighted sum of non-holomorphic Eisenstein series (in the sense of Iwaniec, \cite{I2}, Ch.\ 3).
\begin{defn} Let us define the function $F\colon \uhp\times\{s\in\C:\Re s>1\}\to\C$ as the weighted sum of non-holomorphic Eisenstein series
\begin{align*}F(\tau,s)&\coloneqq\frac{2\pi}{3n}\sum_{\gamma=\begin{bsmallmatrix}\!a&b\!\\\!c&d\!\end{bsmallmatrix}\in\langle T^3\rangle\backslash\Gamma(3)}C(d,-c)(\Im\gamma\tau)^s-C(c,d)(\Im\gamma S\tau)^s
\\&=2\pi(3n)^{s-1}\sum_{\substack{1\,\le\,a,\,b\,\le\,3n\\3\mid a,b-1\\\gcd(a,\,b,\,n)=1}}C(b,-a)E_{[-b/a]\in X(3n)}(\tau,s)-C(a,b)E_{[a/b]\in X(3n)}(\tau,s).\end{align*}
\end{defn}
\begin{lemma}\label{Fisharmonic} The function $\lim_{s\to 1}F(\tau,s)$ is a harmonic function on $\uhp$, and is invariant under the left action on $\uhp$ by $\Gamma(3n)$.
\end{lemma}
\begin{proof} To prove that $F(\tau,s)$ converges uniformly on the upper half-plane in the limit $s\to1$, we may simply use the Fourier expansion for $F(\tau,s)$ provided in Proposition \ref{eisenstein}, and note that the Fourier coefficients converge in the limit $s\to1$ (this follows for the lowest order term from properties of the Riemann zeta function, and for higher order terms from Proposition \ref{phiisholomorphic}).  The fact that $\lim_{s\to1}F(\tau,s)$ is invariant under the left action of $\uhp$ by $\Gamma(3n)$ follows from the fact that the non-holomorphic Eisenstein series for the group $\Gamma(3n)$ are evidently invariant under that action.  Finally, since the Eisenstein series $E_{t\in X(N)}(\tau,s)$ is an eigenfunction of the hyperbolic Laplacian $\Delta$ with eigenvalue $s(1-s)$ (see \cite{I2}), we conclude that $\lim_{s\to1}F(\tau,s)$ is a harmonic function, as desired.
\end{proof}
\begin{thm}\label{harmonicfns} Let $n$ be a positive integer relatively prime to $6$ such that $n\not\equiv\pm1\!\pmod{9}$ and no exponent in the prime factorization of $n$ is a multiple of $3$. 
 For all $\tau\in\uhp$,
\[\log{\abs*{U(\tau)}}=\lim_{s\to1}F(\tau,s)+\frac{f(n)\log 3}{2}.\]
\end{thm}
\begin{proof} We will prove that the harmonic functions on both sides of this equation are equal by showing that their difference is a bounded harmonic function on $\uhp^*$, which must be constant.  For any cusp $[\alpha/\beta]\in X(6n)$, choose a matrix $\gamma\in\SL_2\Z$ such that $[\alpha/\beta]=\gamma\infty$, and let $q_{[\alpha/\beta]}\coloneqq e^{\pi i\gamma^{-1}\tau/3n}$.  If we let $C$ be the leading coefficient of the Fourier expansion of $U(\tau)$ around $\alpha/\beta$, we find by Lemma \ref{matrixtransform}\ref{gleadingcoeff} and Lemma \ref{xplusone} that
\begin{align*}\log{\abs*{U(\tau)}}&=-\frac{\pi\Im\gamma^{-1}\tau}{3n}\ord_{[\alpha/\beta]\in X(6n)}U(\tau)+\log{\abs*{C}}+O(q_{[\alpha/\beta]})\\&=\begin{cases}\frac{2\pi}{3n}C(\alpha,\beta)\Im\gamma^{-1}\tau+O(q_{[\alpha/\beta]})&\quad\text{if }3\mid\beta\\-\frac{2\pi}{3n}C(\alpha,\beta)\Im\gamma^{-1}\tau+f(n)\log3+O(q_{[\alpha/\beta]})&\quad\text{if }3\mid\alpha\\\frac{1}{2}f(n)\log 3+O(q_{[\alpha/\beta]})&\quad\text{otherwise}\,,
\end{cases}
\end{align*}
because $f(n)$ is the number of cusps of $X(6n)$ mapping to any single cusp of $X(6)$.  By Proposition \ref{eisenstein}, $\lim_{s\to1}F(\tau,s)$ has the same singular points with the same residues as $\log{\abs*{U(\tau)}}$, hence $\log{\abs*{U(\tau)}}-\lim_{s\to1}F(\tau,s)$ is a bounded harmonic function on $\uhp^*$, which must be constant.

To determine this constant, we compute the second term in the Fourier expansion of $\lim_{s\to1}F(\tau,s)$ in Proposition \ref{eisenstein} around the cusp $\infty\in\uhp^*$.  We find that
\begin{align*}\lim_{s\to1}F(\tau,s)&=\frac{2\pi}{3n}C(1,0)\Im\tau-4\pi^2\lim_{s\to1}\frac{\zeta(2s-1)}{\zeta(2s)}\\&\times\raisebox{-0.7em}{$\left[\rule{0em}{2.5em}\right.$}\sum_{\substack{1\le a,b\le 3n\\3\mid a,b-1\\\gcd(a,b,n)=1}}\frac{C(a,b)}{(3n)^3}\prod_{p\mid3n}\big(1-p^{-2}\big)^{-1}\bigg[1+(1-s)\bigg(2\log\lambda_b+3\log3n\\[-2em]&\hspace{14.4em}+\sum_{p\mid3n}\frac{2}{p^2-1}\log p+\sum_{p\mid3n/\lambda_b}\frac{2}{p-1}\log p\bigg)\!\bigg]\\[0.5em]&\hspace{1.6111em}-\br \sum_{\substack{1\le a,b\le 3n\\3\mid a,b-1\\\gcd(a,b,n)=1}}\frac{C(b,-a)}{(3n)^3}\prod_{p\mid3n}\big(1-p^{-2}\big)^{-1}\bigg[1+(1-s)\bigg(2\log\lambda_a+3\log3n\\[-3em]&\hspace{14.4em}+\sum_{p\mid3n}\frac{2}{p^2-1}\log p+\sum_{p\mid3n/\lambda_a}\frac{2}{p-1}\log p\bigg)\!\bigg]\!\!\left.\rule{0em}{2.5em}\right]\\&+O(e^{-2\pi\Im\tau}),\end{align*}
where $\lambda_r\coloneqq\gcd(3n,r)$.  We can combine the two series in the following manner.  For any $a,b\in\Z/3n\Z$ such that $3\mid a, b-1$, let $a^*, b^*\in\Z/3n\Z$ such that $3\mid b^*, a^*+1$, $a^*\equiv a\!\pmod{n}$, and $b^*\equiv b\!\pmod{n}$; then $\lambda_{b^*}=3\lambda_b$ and $C(a^*,b^*)=C(a,b)$.  By summing together the $(a,b)$ term in the first sum and the $(b^*, -a^*)$ term in the second term for each $1\le a,b\le 3n$ such that $\gcd(a,b,n)=1$ and $3\mid a,b-1$, we deduce that
\[\lim_{s\to1}F(\tau,s)=\frac{2\pi}{3n}C(1,0)\Im\tau-\frac{12\log3}{(3n)^3}\prod_{p\mid 3n}\big(1-p^{-2}\big)^{-1}\br\sum_{\substack{1\le a,b\le 3n\\3\mid a,b-1\\\gcd(a,b,n)=1}}\br C(a,b)+O(e^{-2\pi\Im\tau}).\]
But if we consider $U$ as a function $X(6n)\to\mathbb{P}^1$, we see that
\[\sum_{\substack{1\le a,b\le 3n\\3\mid a,b-1\\\gcd(a,b,n)=1}}\br C(a,b)=\frac{1}{2}\deg U=|B_0|[\Gamma(6):\Gamma(6n)]=f(n)n^3\prod_{p\mid n}\big(1-p^{-2}\big),\]
hence \[\lim_{s\to1}F(\tau,s)=\frac{2\pi}{3n}C(1,0)\Im\tau-\frac{f(n)\log3}{2}+O(e^{-2\pi\Im\tau}),\]
as desired.\end{proof}
What now remains is to express the special value $\lim_{s\to1}F(\omega,s)$ in terms of the residue of a Dirichlet $L$-function for $k$ associated to a cubic character, which can be rewritten in terms of the Dedekind zeta function of $K$, to which we can apply the class number formula.  Indeed, we have the following theorem.
\begin{thm}\label{lfunction}Let $n$ be a positive integer relatively prime to $6$ such that $n\not\equiv\pm1\!\pmod{9}$ and no exponent in the prime factorization of $n$ is a multiple of $3$. 
 Let $n'$ be the largest squarefree divisor of $n$.  For any $s\in\C$ with $\Re s>1$,
\[F(\omega,s)=\frac{2\pi n'\sigma(n/n')}{3(1-3^{-2s})}\bigg(\!\frac{\sqrt{3}}{2}\bigg)^s\frac{L_k(s,\chi)}{\zeta(2s)}\prod_{p\mid n}\frac{1-p^{-2}}{1-p^{-2s}},\vspace{-0.6em}\]
where $L_k(s,\chi)$ is the Dirichlet L-function with character $\chi\colon(\ok/3n\ok)^\times\to\C^\times$ such that $\chi(z)=\prod_i\!\big(\!\frac{z}{\lambda_i}\!\big)_{\!3}^{\!c_i}$.
\end{thm}
\begin{proof}Since $\chi(\omega)\ne1$, we see that
\begin{align*}F(\omega,s)&=\frac{2\pi}{3n}\sum_{\substack{3\mid c,d-1\\\gcd(c,d)=1}}[C(d,-c)-C(c,d-c)]\frac{(\sqrt{3}/2)^s}{\abs{c\omega+d}^{2s}}
\\&=\frac{4\pi}{9n}\sum_{\substack{3\mid c,d-1\\\gcd(c,d)=1}}\Re\bigg[\!\prod_i\bigg(\!\frac{c\omega+d}{\lambda_i}\!\bigg)_{\!3}^{\!c_i}-\prod_i\bigg(\!\frac{(c-d)\omega+c}{\lambda_i}\!\bigg)_{\!3}^{\!c_i}\bigg]\sum_{n'\mid d\mid n}dn\prod_{p\mid d}\big(1-p^{-2}\big)\frac{(\sqrt{3}/2)^s}{\abs{c\omega+d}^{2s}}\\&=\frac{4\pi n'\sigma(n/n')}{9}\bigg(\!\frac{\sqrt3}{2}\bigg)^{\!s}\prod_{p\mid n}\big(1-p^{-2}\big)\!\!\sum_{\substack{3\mid c, d-1\\\gcd(c,d)=1}}\!\!\Re\bigg[\bigg(\!1-\prod_i\Big(\!\frac{\omega}{\lambda_i}\!\Big)_{\!3}^{\!c_i}\bigg)\prod_i\bigg(\!\frac{c\omega+d}{\lambda_i}\!\bigg)_{\!3}^{\!c_i}\bigg]\frac{1}{\abs{c\omega+d}^{2s}}\\&=\frac{2\pi n'\sigma(n/n')}{3}\bigg(\!\frac{\sqrt3}{2}\bigg)^{\!s}\prod_{p\mid n}\big(1-p^{-2}\big)\cdot\zeta(2s)^{-1}\prod_{p\mid 3n}\big(1-p^{-2s}\big)^{-1}\sum_{3\mid c,d-1}\frac{\chi(c\omega+d)}{\abs{c\omega+d}^{2s}}\\&=\frac{2\pi n'\sigma(n/n')}{3(1-3^{-2s})}\bigg(\!\frac{\sqrt3}{2}\bigg)^{\!s}\frac{L_k(s,\chi)}{\zeta(2s)}\prod_{p\mid n}\frac{1-p^{-2}}{1-p^{-2s}},\end{align*}
as desired.
\end{proof}
\begin{proof}[Proof of Theorem \ref{fundunitcube}]Since $\zeta_K(s)=\zeta(s)L_k(s,\chi)$, we find that by Theorems \ref{harmonicfns} and \ref{lfunction},
\begin{align*}\log{N_{R_{6n}/K}(X(n\omega)+1)}&=2\log{\abs{N_{R_{6n}/kK}(X(n\omega)+1)}}\\&=2\log{\abs{U(\omega)}}\\&=2\lim_{s\to1}F(\omega,s)+f(n)\log3\\&=\frac{9\sqrt{3}n'\sigma(n/n')}{2\pi}\Res_{s=1}\zeta_K(s)+f(n)\log3.\end{align*}
By the class number formula,
\[\Res_{s=1}\zeta_K(s)=\frac{2\pi h_K\log u}{3n'\sqrt{3}},\]
and the theorem follows.\end{proof}\vspace{0.3em}
\section{\label{heegner}Heegner Points on $y^2=x^3+D$.}
We now turn to using Heegner points to construct rational points on the curve $E_D\colon y^2=x^3+D$ and proving, under appropriate conditions on $D$, that they are nontrivial, by applying Theorem \ref{fundunitcube}.  Recall from Theorem \ref{themodularparam} that we constructed modular functions $X$ and $Y$ of weight $0$ for $\Gamma(6)$ such that $Y^2=X^3+1$; these functions form a modular parametrization $\phi\coloneqq(X,Y)\colon X(6)\isom E_1$.

Suppose $a$ and $b$ are squarefree integers such that $6$, $a$, and $b$ are pairwise relatively prime and $\abs{a}\abs{b}^{-1}\equiv5$ or $7\!\pmod{9}$, and let $D\coloneqq a/b$.  Let us define $n\coloneqq\abs{ab^5}$, $\rho\coloneqq(-1)^{(n-1)/2}$, and $K\coloneqq\Q(\!\sqrt[3]{n})$; then $E_{\epsilon D}\cong E_{\rho n}$ as elliptic curves over $\Q$ (where $\epsilon\coloneqq(-1)^{(D-1)/2}$ is defined in Theorem \ref{mainresult}).  For any integers $A$, $B$, and $N$, consider the elliptic curves $C_{A,B}$ with (affine) equation $AX^3+BY^3=1$ and identity element $[1/\sqrt[3]{A}:-1/\sqrt[3]{B}:0]$ and $\widetilde{E}_N$ with equation $X^3+Y^3=N$ and identity element $[1:-1:0]$.  We can then construct the following isogenies, the second of which is described by Selmer \cite{Sel}, who attributes a version of it to Euler.\newpage

\begin{prop}\begin{enumerate}[label=\textup{(\alph*)}, leftmargin=*]\item The map $\varphi_N\colon E_{-27N^2}\to\widetilde{E}_{2N}$ defined by
\[\varphi_N(x,y)\coloneqq\Big(\frac{9N+y}{3x}, \frac{9N-y}{3x}\Big)\]
is an isomorphism of elliptic curves with inverse
\[\varphi_N^{-1}=\Big(\frac{6N}{X+Y}, \frac{9N(X-Y)}{X+Y}\Big).\]
\item\textup{\cite{Sel}} The morphism $\lambda_{A,B}\colon C_{A,B}\to E_{AB}$ defined by $\lambda_{A,B}(x,y)\coloneqq\big(\frac{U+V}{2}, \frac{U-V}{2}\big)$ where
\[U\coloneqq\frac{3ABx^2y^2}{1-ABx^3y^3}\]and\[V\coloneqq\frac{(Ax^3-By^3)(2+ABx^3y^3)}{3xy(1-ABx^3y^3)}\]is an isogeny, and has the property that no point in $\ker\lambda_{A,B}$ is defined over $\Q$.\end{enumerate}
\end{prop}
\noindent Our strategy to construct points on $E_{\rho n}$ (and consequently on $E_{\epsilon D}$) is then to consider a point on $E_1$ arising as the image of a Heegner point on $X(6)$, and find its image under the following chain of isogenies and isomorphisms. The maps in the diagram below marked with an `$\cong$' scale $x$ and $y$ by a constant factor.
\[\begin{tikzpicture}
    \draw [->] (-0.95, 0) -- (0.4, 0);
    \draw [->] (1.35, 0) -- (2.525, 0);
    \draw [->] (3.05, 0.1) -- (4.525, 0.4);
    \draw [->] (3.075, -0.05) -- (4.465, -0.35);
    \draw [->] (5.5, 0.375) -- (6.7, 0.075);
    \draw [->] (5.5, -0.35) -- (6.7, -0.075);
    \draw [->] (7.7, 0) -- (8.6, 0);
    \draw [->] (9.875, 0) -- (11.25, 0);
    \draw [->] (12.5, 0) -- (13.475, 0);
    \node at (-1.225, 0.005) {$E_1$};
    \node at (0.875, -0.01) {$E_{-27}$};
    \node at (2.8, 0.055) {$\widetilde{E}_2$};
    \node at (5, 0.4) {$C_{n^2, 2}$};
    \node at (5, -0.45) {$C_{1, 2n^2}$};
    \node at (7.2, 0.06) {$\widetilde{E}_{2n^2}$};
    \node at (9.25, -0.015) {$E_{-27n^4}$};
    \node at (11.9, -0.02) {$E_{-27\rho n}$};
    \node at (13.825, -0.02) {$E_{\rho n}$};
    \node at (-0.275, 0.15) {\scriptsize$\cong\!/k$};
    \node at (1.9375, 0.15) {\scriptsize$\varphi_1$};
    \node at (3.775, 0.515) {\scriptsize$\cong\!\!/\!K\!\big(\!\raisebox{-0.1em}{$\sqrt[3]{2}$}\big)$};
    \node at (3.775, -0.49) {\scriptsize$\cong\!\!/\!K\!\big(\!\raisebox{-0.1em}{$\sqrt[3]{2}$}\big)$};
    \node at (6.05, 0.425) {\scriptsize$\lambda_{n^2,2}$};
    \node at (6.125, -0.475) {\scriptsize$\lambda_{1,2n^2}$};
    \node at (8.15, 0.25) {\scriptsize$\varphi_{n^2}^{-1}$};
    \node at (10.5625, 0.175) {\scriptsize$\cong\!\!/\Q(\!\sqrt{\vphantom{b}\!\rho n})$};
    \node at (13, 0.15) {\scriptsize$\cong\!/k$};
\end{tikzpicture}\]
The first step is to show that we can use our point $\phi(n\omega)\in E_1(R_{6n})$ to obtain a point in either $C_{n^2,2}(R_n)$ or $C_{1,2n^2}(R_n)$, depending on whether $n\equiv 5$ or $7\!\pmod{9}$.  The image of $\phi(n\omega)$ in $\smash{\widetilde{E}_2(R_{6n})}$ is $(f_-(n\omega), -f_+(n\omega))$, where
\[f_\pm(\tau)\coloneqq\frac{Y(\tau)\pm\sqrt{-3}}{X(\tau)\sqrt{-3}}.\]
Let us now determine the action of $\Gal(R_{6n}/R_n)$ on $(f_-(n\omega), -f_+(n\omega))$.  The following lemma is a generalization of Proposition 2.16 in \cite{Sa}.
\begin{lemma}\label{randj} Let $x\in\Z[n\omega]$, and let $r,j\in\Z$ such that $x\equiv\omega^r\!\pmod{2}$ and $x\equiv\omega^j\!\pmod{3}$.  Then $S(x)\equiv(\gamma')^r\gamma_{-\hspace{-0.11em}\left(\hspace{-0.22em}\frac{n}{3}\hspace{-0.22em}\right)}^j\!\pmod{6}$, where $\gamma_\pm$ and $\gamma'$ are matrices satisfying the conditions in Proposition \ref{modparamidentities}.
\end{lemma}
\begin{prop} Let $x\in\Z[n\omega]$, and let $r$ and $j$ be as in Lemma \ref{randj}.  Then
\[(x,\,R_{6n}/k)(f_-(n\omega), -f_+(n\omega))=\big(\omega^{r-\left(\hspace{-0.22em}\frac{n}{3}\hspace{-0.22em}\right)j}f_-(n\omega), -\omega^{-r-\left(\hspace{-0.22em}\frac{n}{3}\hspace{-0.22em}\right)j}f_+(n\omega)\big)+Q_j,\]
where in projective coordinates,
\[Q_j\coloneqq\begin{cases}[1:-1:0]&\quad j\equiv0\br\pmod{3}\\ [\omega^{-\left(\hspace{-0.22em}\frac{n}{3}\hspace{-0.22em}\right)}:\omega^{-\left(\hspace{-0.22em}\frac{n}{3}\hspace{-0.22em}\right)}:1]&\quad j\equiv1\br\pmod{3}\\ [1:1:1]&\quad j\equiv2\br\pmod{3}\,.\end{cases}\]\end{prop}
\begin{proof} By Proposition \ref{modparamidentities},
\[\big(f_-(\gamma_{-\hspace{-0.11em}\left(\hspace{-0.22em}\frac{n}{3}\hspace{-0.22em}\right)}^j\tau), -f_+(\gamma_{-\hspace{-0.11em}\left(\hspace{-0.22em}\frac{n}{3}\hspace{-0.22em}\right)}^j\tau)\!\big)=\big(\omega^{-j\left(\hspace{-0.22em}\frac{n}{3}\hspace{-0.22em}\right)}f_-(\tau), -\omega^{-j\left(\hspace{-0.22em}\frac{n}{3}\hspace{-0.22em}\right)}f_+(\tau)\!\big)+Q_j\]
and
\[(f_-(\gamma'\tau), -f_+(\gamma'\tau)\!)=(f_-(\tau), -f_+(\tau)\!)+[\omega:-\omega^{-1}:0]=(\omega f_-(\tau), -\omega^{-1}f_+(\tau)\!).\]
Combining these two results with Lemma \ref{randj} proves the proposition.\end{proof}
\begin{thm}\begin{enumerate}[label=\textup{(\alph*)}, leftmargin=*]\item\label{5mod9} Let $g\colon\widetilde{E}_2\xrightarrow{\raisebox{-0.2em}{\scriptsize$\;\cong\;$}}C_{n^2,2}$ be the map with $g(x,y)=(x/\sqrt[3]{2n^2}, y/\sqrt[3]{4})$.  If $n\equiv5\!\pmod{9}$,
\[g((f_-(n\omega), -f_+(n\omega)\!)+(\omega^{-1},\omega^{-1}))\in C_{n^2, 2}(R_n).\]
\item\label{7mod9} Let $h\colon\widetilde{E}_2\xrightarrow{\raisebox{-0.2em}{\scriptsize$\;\cong\;$}}C_{1,2n^2}$ be the map with $h(x,y)=(x/\sqrt[3]{2}, y/\sqrt[3]{4n^2})$.  If $n\equiv7\!\pmod{9}$,
\[h((f_-(n\omega), -f_+(n\omega)\!)+(\omega,\omega))\in C_{1, 2n^2}(R_n).\]
\end{enumerate}
\end{thm}
\begin{proof} We will prove part \ref{5mod9}, the proof of part \ref{7mod9} is similar.  We need to show that $g((f_-(n\omega), -f_+(n\omega)\!)+(\omega^2,\omega^2))$ is invariant under $(x,\,R_{6n}/R_n)$ for all $x\in\Z[n\omega]$.  Let $\psi\colon C_{n^2,2}\to C_{n^2,2}$ be the isogeny with $\psi(x,y)=(\omega x,\omega y)$.  For any integers $a$ and $b$, and any $(x,y)\in C_{n^2,2}(\overline{\Q})$,
\[(\omega^ax, \omega^by)=\psi^{-(a+b)}((x,y)+g[\omega^{-a+b}:\omega^{a-b}:0]).\]
Let $A\coloneqq(\omega^{-1}, \omega^{-1})\in\widetilde{E}_2(\overline{\Q})$.  If $\theta\colon\widetilde{E}_2\to\widetilde{E}_2$ is the isogeny with $\theta(x,y)=(\omega x, \omega y)$, 
\begin{align*} &\hspace{1.3333em}(x,\,R_{6n}/k)[g((f_-(n\omega), -f_+(n\omega))+A)]\\&=((x,\,R_{6n}/k)g)((\omega^{r+j}f_-(n\omega), -\omega^{-r+j}f_+(n\omega))+Q_j+A)\\&=(\psi^{-j}\circ g)((\omega^{r+j}f_-(n\omega), -f_+(n\omega))+[\omega^{-r}:\omega^r:0]+Q_j+A)\\&=g((f_-(n\omega), -f_+(n\omega))+(\psi^{-j}\circ g)(Q_j+A)\\&=g((f_-(n\omega), -f_+(n\omega))+\theta^{-j}(Q_j+A).
\end{align*}The theorem follows from the fact that $A=\theta^{-j}(Q_j+A)$, which can be shown by casework on $j\!\pmod{3}$.\end{proof}
\noindent Next, we can now follow through the rest of the diagram, and conclude that if $n\equiv5$ or $7\!\pmod{9}$,
\[(a_n(n\omega),b_n(n\omega))+\big({-}\omega^{-\left(\hspace{-0.22em}\frac{n}{3}\hspace{-0.22em}\right)}\rho\sqrt[3]{n},0\big)\in E_{\rho n}(R_n)\]
where
\[a_n(\tau)\coloneqq-\frac{\rho\sqrt[3]{n}}{3}\bigg(\frac{4+f_+(\tau)^3f_-(\tau)^3}{f_+(\tau)^2f_-(\tau)^2}\bigg)\]
and
\[b_n(\tau)\coloneqq\frac{\sqrt{\rho n}}{6\sqrt{-3}}\bigg(\frac{(f_+(\tau)^3+f_-(\tau)^3)(8-f_+(\tau)^3f_-(\tau)^3)}{f_+(\tau)^3f_-(\tau)^3}\bigg).\]
Then we may construct a point on $E_{\rho n}$ defined over $\Q$ as
\begin{equation}\label{thepointS}S^*\coloneqq\Tr_{R_n/\Q}\big(\!(a_n(n\omega),b_n(n\omega))+\big({-}\omega^{-\left(\hspace{-0.22em}\frac{n}{3}\hspace{-0.22em}\right)}\rho\sqrt[3]{n},0\big)\!\big)\in E_{\rho n}(\Q).\end{equation}
To prove that this point is nontrivial, we will use the following homomorphism, described for all elliptic curves $E_D$ by Cassels \cite{Cas} and in greater generality by Silverman \cite{Silv1}, to relate the nontriviality of $S^*$ to a certain element of $R_{6n}^\times$ not being a square.
\begin{lemma}\label{grouphom}\textup{\cite{Cas, Silv1}} For any number field $L\supseteq K$, there is a group homomorphism
\[r_L\colon E_{\rho n}(L)/[2]E_{\rho n}(L)\to L^\times/(L^\times)^2\]
such that $r_L(x,y)=x+\!\rho\sqrt[3]{n}$ for all $(x,y)\not\in E_{\rho n}(L)[2]$.  This group homomorphism has the special values $r_L(O)=1$ and $r_L({-}\rho\sqrt[3]{n}, 0)=3$.
\end{lemma}
\noindent Under this homomorphism, we see that if $L=R_{6n}$,
\begin{align*}&\hspace{1.3333em}r_L\big(\!(a_n(n\omega),b_n(n\omega))+\big({-}\omega^{-\left(\hspace{-0.22em}\frac{n}{3}\hspace{-0.22em}\right)}\rho\sqrt[3]{n},0\big)\!\big)\\&=-\Big(\hspace{-0.0555em}\frac{n}{3}\hspace{-0.0555em}\Big)\frac{X(n\omega)+1}{\sqrt{-3}}\bigg(\frac{\omega^{-\left(\hspace{-0.22em}\frac{n}{3}\hspace{-0.22em}\right)}\!\sqrt[3]{n}\,(f_+(n\omega)f_-(n\omega)-2)(X(n\omega)-2)}{f_+(n\omega)f_-(n\omega)X(n\omega)\sqrt{-3}}\bigg)^{\!2}\\&\equiv-\Big(\hspace{-0.0555em}\frac{n}{3}\hspace{-0.0555em}\Big)\sqrt{-3}\,(X(n\omega)+1)\in L^\times/(L^\times)^2.\end{align*}
\begin{proof}[Proof of Theorem \ref{mainresult}] Now $E_{\pm D}(\Q)_{tors}$ is trivial, and by a result of Cassels \cite{Cas}, $\rk E_D(\Q)+\rk E_{-D}(\Q)\le 1$ under our congruence and class number conditions on $D$.  To prove that $\rk E_{\epsilon D}(\Q)=1$ and $\rk E_{-\epsilon D}(\Q)=0$, it suffices to show that $S^*\in E_{\rho n}(\Q)$ is nontrivial.  In order to accomplish this, we will use Lemma \ref{grouphom} when $L=R_{6n}$ and prove that $r_L(S^*)\in L^\times/(L^\times)^2$ is not a square.  Since $n\not\equiv\pm1\!\pmod{9}$, $L=R_n(\!\sqrt[3]{2}, \sqrt[3]{n\vphantom{2}})$, and
\begin{align*}r_L([3]S^*)&=N_{R_{3n}/K}\,r_L\big(\!(a_n(n\omega),b_n(n\omega))+\big({-}\omega^{-\left(\hspace{-0.22em}\frac{n}{3}\hspace{-0.22em}\right)}\rho\sqrt[3]{n},0\big)\!\big)^3\\&=N_{L/K}(X(n\omega)+1)=3^{f(n)}u^{3h_K\sigma(n/n')}\equiv u\in L^\times/(L^\times)^2,\end{align*}
so it remains for us to show that $u$ is not a square in $L^\times$.  But since $h_K$ is odd, $\sqrt{u}$ cannot possibly be contained in the Hilbert class field of $K$; the field extension $K(\sqrt{u})/K$ must be ramified.  Since $u$ is a unit, $\sqrt{u}$ could only possibly be ramified over $K$ at a place dividing $2$.  Then $\sqrt{u}\notin R_{3n}$, as $R_{3n}$ is necessarily unramified over $\Q$ at all places not dividing $3n$, and hence $\sqrt{u}\notin L$ because $[L:R_{3n}]=3$, as desired.
\end{proof}
\vspace{-0.24em}In fact, this proof gives us further insights.  Since $r_L(S^*)\in L^\times/(L^\times)^2$ is nontrivial, we find that $S^*$ must be an odd multiple of the generator of $E_{\rho n}(\Q)$, which proves Theorem \ref{rankgenerator}.
\section*{Acknowledgments}
I am very grateful to Prof.\ Barry Mazur for his invaluable advice and guidance to help me gain an understanding of and pursue research in elliptic curves.  I thank Prof.\ Noam D.\ Elkies for insights on Heegner points and for numerically verifying my main result for all positive integers $D<5000$.  I thank Prof.\ Wei Zhang for illuminating me on possible directions and current research in the subject.  I thank Prof.\ Joe Harris for his kind support and encouragement.  I also thank Prof.\ Alexander Betts, Dr.\ Simon Rubinstein-Salzedo, and Dr.\ Alex Cowan for helpful discussions.

\end{document}